\documentclass{amsart}

\usepackage{amsmath, amssymb, amsthm, amsfonts}
\usepackage{adjustbox}
\usepackage{orcidlink}
\usepackage{mathtools}
\usepackage{geometry}
\usepackage{hyperref}
\definecolor{vegasgold}{rgb}{0.77, 0.7, 0.35}
\definecolor{darkgoldenrod}{rgb}{0.72, 0.53, 0.04}
\definecolor{gold(metallic)}{rgb}{0.83, 0.69, 0.22}
\hypersetup{
 colorlinks=true,
 linkcolor=darkgoldenrod,
 filecolor=brown,      
 urlcolor=gold(metallic),
 citecolor=darkgoldenrod,
 }

\usepackage{graphicx}
\usepackage{tikz-cd}
\usepackage{enumitem}

\newcommand{\Q}{\mathbb{Q}}

\newcommand{\F}{\mathbb{F}}
\newcommand{\op}[1]{\operatorname{#1}}
\newcommand{\R}{\mathbb{R}}
\newcommand{\Z}{\mathbb{Z}}
\newcommand{\GL}{\mathrm{GL}}
 \newcommand{\AR}[1]{\textcolor{blue}{#1}}

\theoremstyle{plain}
\newtheorem{theorem}{Theorem}[section]
\newtheorem{lemma}[theorem]{Lemma}
\newtheorem{proposition}[theorem]{Proposition}
\newtheorem{corollary}[theorem]{Corollary}

\theoremstyle{definition}
\newtheorem{definition}[theorem]{Definition}

\theoremstyle{remark}
\newtheorem{remark}[theorem]{Remark}

\numberwithin{equation}{section}

\title[Shapes of multi-quadratic number fields]
{On the distribution of shapes of totally real multiquadratic number fields}

\author[A.~Jakhar]{Anuj Jakhar\, \orcidlink{0009-0007-5951-2261}}
\address{Indian Institute of Technology Madras, Chennai, India}
\email{anujjakhar@iitm.ac.in}

\author[A.~Ray]{Anwesh Ray\, \orcidlink{0000-0001-6946-1559}}
\address{Chennai Mathematical Institute, Chennai, India}
\email{anwesh@cmi.ac.in}


\begin{document}

\maketitle

\begin{abstract}
The \emph{shape} of a number field $K$ of degree $m$ is defined as the
equivalence class of the lattice of integers under linear operations
generated by rotations, reflections, and positive scalar dilations.
It may be viewed as a point in the space of shapes
\[
\mathcal{S}_{m-1}
=
\GL_{m-1}(\Z)\backslash \GL_{m-1}(\R)/\op{GO}_{m-1}(\R).
\] The double quotient space is equipped with a  natural measure $\mu$ which is induced from the Haar measure on $\GL_{m-1}(\R)$. We study the distribution of shapes of totally real multiquadratic number fields of degree $m:=2^n$ in which $2$ is unramified. We show that the distribution is governed by the restriction of $\mu$ to a certain torus orbit in $\mathcal{S}_{m-1}$. Our result resolves a conjecture of Haidar.
\end{abstract}

\section{Introduction}

\subsection{Motivation and historical context} 
A classical approach to studying the arithmetic of a number field is to translate its algebraic properties into geometric ones. If we consider a number field $K$ of degree $n = [K:\mathbb{Q}]$, a standard tool for this translation is the Minkowski embedding: $$J: K \hookrightarrow K_{\mathbb{R}} := K \otimes_{\mathbb{Q}} \mathbb{R} \cong \mathbb{R}^n.$$ This embedding realizes $\mathcal{O}_K$ as a full-rank Euclidean lattice $\Lambda_K = J(\mathcal{O}_K)$ inside $K_{\mathbb{R}}$. Much of classical algebraic number theory relies on a single invariant to understand this lattice: the absolute discriminant $\Delta_K$. Since the square root of the discriminant computes the covolume of $\Lambda_K$, it effectively measures the ``size'' of the ring. Yet, a single volume measurement completely ignores the angles between generators of the lattice. Two fields can have identical discriminants but  different lattice structures, with different angles and varying lengths of basis vectors. To capture these exact geometric proportions, we turn to a much sharper invariant known as the {\it shape} of the number field.

To define the shape rigorously, we view $K_{\mathbb{R}}$ as a Euclidean space equipped with the inner product induced by the algebraic trace form. Since $\Lambda_K$ always contains the distinguished vector $J(1) = (1, \dots, 1)$, this introduces a rigid linear constraint common to all number fields. It is therefore standard to project the lattice orthogonally onto the trace-zero hyperplane. The shape of $K$ is defined as the equivalence class of this projected lattice of rank $n-1$, considered modulo the group $\op{GO}_{n-1}(\mathbb{R})$ of orthogonal similitudes (the group generated by rotations, reflections, and uniform scalar dilations). The space of all such shapes is identified with the double coset space
\[
    \mathcal{S}_{n-1} := \op{GL}_{n-1}(\mathbb{Z}) \backslash \op{GL}_{n-1}(\mathbb{R}) / \op{GO}_{n-1}(\mathbb{R}),
\]
which comes equipped with a natural measure $\mu$ derived from the Haar measure on $\op{GL}_{n-1}(\mathbb{R})$.

A central pursuit in arithmetic statistics is understanding how these shapes are distributed as $K$ varies over a family of number fields of a fixed degree, ordered by absolute discriminant. For generic families, namely, fields whose Galois closure has Galois group isomorphic to $S_n$, it is expected that the shapes equidistribute across the entire space $\mathcal{S}_{n-1}$ with respect to the measure $\mu$. This principle of generic equidistribution has been beautifully verified in low degrees. Terr \cite{Terr97} established that the shapes of cubic fields are equidistributed in the modular surface $\mathcal{S}_2$. Subsequently, Bhargava and Harron \cite{BH16} extended this phenomenon to $S_4$-quartic and $S_5$-quintic fields, utilizing parameterizations of low-degree rings (cf. \cite{HighercompositionIII, HighercompositionIV}).

\par These methods do not apply to non-generic families of number fields, i.e., when the Galois group of the Galois closure is a proper subgroup of $S_n$. In such  families, underlying arithmetic constraints force the shapes to collapse into lower-dimensional subvarieties of $\mathcal{S}_{n-1}$. The first major breakthrough in this direction was Harron's study of pure cubic fields \cite{Har17}. He proved that their shapes do not fill the space $\mathcal{S}_2$, but instead restrict strictly to two specific one-dimensional geodesics, governed entirely by a tame-wild ramification dichotomy at the prime $3$. Holmes \cite{Holmes} later generalized this phenomenon to pure prime-degree fields $\mathbb{Q}(\sqrt[p]{m})$, proving that their shapes lie on affine subspaces of dimension $(p-1)/2$. Similar results have been proven for cyclic prime-degree fields \cite{MSM16}, Galois quartic fields \cite{harronharrongaloisquartic}, and complex cubic fields with a fixed quadratic resolvent \cite{HarronANT}.

Most recently, investigations into non-generic fields of composite degree have revealed even richer geometric structures. Studies on pure quartic fields \cite{purequartic}, pure sextic fields \cite{JKRR26}, and octic Kummer extensions \cite{Octic} have demonstrated that the space of shapes for these families decomposes into finite unions of translated torus orbits.
\par In this paper we study the non-generic family of totally real
multiquadratic number fields of degree $2^n$.  Such a field can be
written
\[
K_n=\Q(\sqrt{D_{2^0}},\dots,\sqrt{D_{2^{n-1}}}),
\]
where the $D_{2^k}$ are distinct squarefree positive integers satisfying
natural non-degeneracy conditions that ensure $[K_n:\Q]=2^n$.
$K_n$.  
The Galois group of $K_n/\Q$ is the elementary abelian
$2$–group $\F_2^n$, and the field contains exactly
$\ell=2^n-1$ quadratic subfields.

The distribution of shapes in multiquadratic families has only recently
begun to be understood.  In the biquadratic case this problem was first
studied by Harron and Harron~\cite{PR20}, who described the space of possible shapes of biquadratic fields.  Subsequently, Haidar
\cite{Hassan} investigated the triquadratic case and formulated a
precise equidistribution conjecture for totally real multiquadratic
fields in which the prime $2$ is unramified. In particular, he conjectured that the shapes of such fields,
when ordered by discriminant, become equidistributed inside the
subspace of the full shape space that they occupy.
\par While we explicitly determine the shapes for all totally real multiquadratic fields regardless of their ramification at $2$, the main density result of this paper focuses on the unramified case, proving Haidar's Conjecture~1.1 for all $n$. Our approach combines an arithmetic parametrization
of multiquadratic fields with analytic counting methods. A key input is
the classification of multiquadratic integral bases due to Chatelain
(\cite{Chatelain}, \cite[Theorem~2.1]{Man}).  
\subsection{Main Results}

Let $n\ge 2$ and set $\ell=2^{n}-1$.  We consider totally real
multiquadratic extensions
\[
K_n=\Q(\sqrt{a_1},\dots,\sqrt{a_n})
\]
of degree $2^n$.  Associated to such a field are $\ell$ quadratic
subfields whose radicands may be arranged into a tuple
$(D_1,\dots,D_\ell)$ obtained by taking squarefree parts of products of
the generators $a_i$.  As in \cite{Hassan}, the shapes of such fields
can be described by $\ell-1$ real parameters $\lambda_2,\dots,\lambda_\ell$ which determine the similarity class of the projected Minkowski
lattice.  We denote by $\mathcal{S}_{K_n}\subset\mathcal{S}_\ell$ the
subset of the space of shapes consisting of those arising from totally
real multiquadratic extensions.

Let $R_2,\dots,R_\ell$ be fixed positive real numbers satisfying
\[
1\le R_2\le R_3\le\dots\le R_\ell .
\]
Following \cite{Hassan}, we define the region
\[
W(R_2,\dots,R_\ell)
=
\left\{
\operatorname{sh}(\lambda_2,\dots,\lambda_\ell)\in\mathcal{S}_{K_n}
:
\lambda_2<\dots<\lambda_\ell\le R_\ell,\;
R_j\le\lambda_j,\quad 2\le j\le\ell-1
\right\}.
\]
For $X>0$ we write
\[
\mathcal{F}(X,R_2,\dots,R_\ell)
=
\#\left\{
K_n :
\Delta_{K_n}\le X,\;
\operatorname{sh}(K_n)\in W(R_2,\dots,R_\ell),\;
2 \text{ unramified in } K_n
\right\}.
\]
\noindent Our first result gives an asymptotic formula for this counting
function.  The main term involves an explicit constant together with an iterated integral
\[
F(R_2,\dots,R_\ell):=\int_{R_{\ell-1}}^{R_\ell}
\!\!\cdots
\int_{R_2}^{a_3}
\frac{da_2}{a_2}\cdots\frac{da_\ell}{a_\ell}.
\]
This is a polynomial in the variables $\log R_j$.

\begin{theorem}\label{thm:main-intro}
Let $n\ge1$ and $\ell=2^n-1$.  Fix parameters
$1\le R_2\le\dots\le R_\ell$.  Then
\[
\mathcal{F}(X,R_2,\dots,R_\ell)
\sim
C_\ell\,
F(R_2,\dots,R_\ell)\,
X^{1/2^{\,n-1}}
\qquad (X\to\infty),
\]
where
\[
C_\ell=
\frac{\omega_1 c_\ell\ell !}{4^\ell\,\#\GL_n(\F_2)}
\prod_{p>2}
\left(1-\frac1p\right)^\ell
\left(1+\frac{\ell}{p}\right).
\]
\end{theorem}

The constant $c_\ell:=|\op{det} C|^{-1}$ for an explicit matrix \eqref{defn of C} which arises from the archimedean volume computation
for the region of shape parameters, while the Euler product reflects
the density of tuples satisfying the strong carefreeness conditions at
each prime. For the definition of $\omega_1$, see Definition \ref{defn of omega}.

Theorem \ref{thm:main-intro} implies that the shapes of these fields
become equidistributed in the subset $\mathcal{S}_{K_n}$ with respect
to the natural measure $\mu$.

\begin{theorem}\label{thm:equidistribution}
Let $\mathcal{M}_n^{+}(i)$ denote the family of totally real
multiquadratic fields of degree $2^n$ in which $2$ is unramified.
Then the shapes of the fields in $\mathcal{M}_n^{+}(i)$ are
equidistributed in $\mathcal{S}_{K_n}$ with respect to $\mu$ in the
regularized sense of \cite[Conjecture~1.1]{Hassan}.  That is, there
exists a constant $C_n>0$ such that for every compact
$\mu$--continuity set $W\subset\mathcal{S}_{K_n}$,
\[
\lim_{X\to\infty}
X^{-1/2^{\,n-1}}
\#\{K_n\in\mathcal{M}_n^{+}(i):
\Delta_{K_n}\le X,\;
\operatorname{sh}(K_n)\in W\}
=
C_n\,\mu(W).
\]
\end{theorem}
Theorem \ref{thm:equidistribution} follows from Theorem \ref{thm:main-intro} since up to a positive constant, $F(R_2, \dots, R_\ell)$ equals $\mu\left(W(R_2, \dots, R_\ell)\right)$ (cf. \cite[section 2.4]{Hassan} for details). This establishes
\cite[Conjecture~1.1]{Hassan} for all $n$.
\par We note here that our arguments, when specialized to $n=3$ lead to several corrections in the unpublished work \cite{Hassan},
particularly concerning the parametrization of volumes and the sieve tail estimate in passing to congruence conditions at all primes.
Specifically, the sieve-theoretic deduction appearing on
\cite[line~6, p.~54]{Hassan} is not valid as stated, since the implied
constant in that estimate depends on the prime $p$ when $p$ is large
relative to $N$. We rectify this issue in the proof of
Theorem~\ref{thm:infinite-sieve-general}.
\section{Preliminary notions} 
\subsection{Shapes of number fields} \label{sec:shapes_notion}
In order to give a formal definition of the shape of a number field, we first define the geometric notion of the shape of a full rank lattice contained in an Euclidean inner product space. Let $V$ be a real vector space of finite dimension $m$, equipped with an inner product $\langle \cdot, \cdot \rangle$. Choose a basis $\mathbf{e}_1, \dots, \mathbf{e}_m$ of $V$ such that 
\[\langle \mathbf e_i, \mathbf e_j\rangle =\begin{cases}
    1 & \text{ if }i=j;\\
    0 & \text{ if }i\neq j.
\end{cases}\] With respect to this choice of basis identify the automorphism group of $V$ with the group of $m\times m$ invertible matrices $\op{GL}_m(\mathbb{R})$. Let $\Lambda \subset V$ be a lattice of full rank. Let $\operatorname{GO}_m(\mathbb{R})$ be the group consisting of $M\in \op{GL}_m(\mathbb{R})$ such that $M M^T$ is a scalar matrix.

\par Choose a $\mathbb{Z}$-basis $ w_1, \dots, w_m$ of $\Lambda$ and write $ w_i=\sum_{j=1}^m a_{i,j} \mathbf e_j$. Consider the matrix \[W =(a_{i,j})\in \operatorname{GL}_m(\mathbb{R}).\] 
\begin{definition}The \emph{space of shapes} of rank-$m$ lattices is defined as the double coset space:
\[
    \mathcal{S}_m := \operatorname{GL}_m(\mathbb{Z}) \backslash \operatorname{GL}_m(\mathbb{R}) / \operatorname{GO}_m(\mathbb{R}),
\]
and the \emph{shape of $\Lambda$}, denoted $\operatorname{sh}(\Lambda)$, is the double coset $\operatorname{GL}_m(\mathbb{Z}) W \operatorname{GO}_m(\mathbb{R}) \in \mathcal{S}_m$. 
\end{definition}
\noindent Note that the shape is independent of the choice of $\mathbb{Z}$-basis for $\Lambda$ and remains invariant if $\Lambda$ is scaled or acted upon by an orthogonal transformation. 
\par An equivalent and computationally convenient description of lattice shapes is formulated using Gram matrices. Let $\mathcal{P}_m$ denote the space of $m \times m$ positive-definite, symmetric real matrices. The group $\operatorname{GL}_m(\mathbb{Z})$ acts on $\mathcal{P}_m$ via $U \cdot S := U S U^T$ for $U \in \operatorname{GL}_m(\mathbb{Z})$. On the other hand, the multiplicative group $\mathbb{R}^\times$ acts by homotheties (scaling), $r \cdot S := r^2 S$ for $r \in \mathbb{R}^\times$. The map sending a basis matrix $W$ to its corresponding Gram matrix $W W^T$ induces a natural, $\operatorname{GL}_m(\mathbb{Z})$-equivariant bijection:
\[
    \operatorname{GL}_m(\mathbb{R}) / \operatorname{GO}_m(\mathbb{R}) \;\longrightarrow\; \mathcal{P}_m / \mathbb{R}^\times.
\]
Consequently, the space of shapes can be canonically identified with the quotient space:
\begin{equation} \label{eq:shape_gram_space}
    \mathcal{S}_m \;=\; \operatorname{GL}_m(\mathbb{Z}) \backslash \mathcal{P}_m / \mathbb{R}^\times.
\end{equation}

In this framework, if the lattice $\Lambda$ has a $\mathbb{Z}$-basis $w_1, \dots, w_m$, its Gram matrix is denoted by \[\operatorname{Gr}(\Lambda) := \bigl( \langle w_k, w_j \rangle \bigr)_{k,j}.\] Under the identification in \eqref{eq:shape_gram_space}, the shape of $\Lambda$ is precisely the equivalence class of $\operatorname{Gr}(\Lambda)$ in $\mathcal{S}_m$. Indeed, if one applies a unimodular change of basis $U \in \operatorname{GL}_m(\mathbb{Z})$ and a scaling factor $r \in \mathbb{R}^\times$, the Gram matrix transforms to $r^2 U \operatorname{Gr}(\Lambda) U^T$. This demonstrates that the equivalence class in $\mathcal{S}_m$ is invariant with respect to both transformations.

\par Now let $K$ be a number field of degree $n$ with distinct embeddings $\sigma_1, \dots, \sigma_n : K \hookrightarrow \mathbb{C}$. Let $K_{\mathbb{R}}$ be the real vector space spanned by the image of the Minkowski embedding
\[J: K\rightarrow \mathbb C^n\] defined by $J(x):=(\sigma_1(x), \dots, \sigma_n(x))$. This gives an inclusion $J: K\hookrightarrow K_{\mathbb{R}}$. The Hermitian inner product on $\mathbb{C}^n$ restricts to a symmetric inner product on $K_{\mathbb{R}}$. For any $x, y \in K$, this inner product evaluates to
\[
    \langle J(x), J(y) \rangle = \sum_{i=1}^n \sigma_i(x)\overline{\sigma_i(y)}.
\]
We embed $K$ into $K_{\mathbb{R}}$ via the canonical Minkowski embedding:
\[
    J : K \hookrightarrow K_{\mathbb{R}}, \quad x \mapsto \bigl(\sigma_1(x), \ldots, \sigma_n(x)\bigr).
\]
Under the mapping $J$, the ring of integers $\mathcal{O}_K$ forms a full-rank lattice in $K_{\mathbb{R}}$. A natural first attempt to define the shape of $K$ would be to take the shape of this full lattice $J(\mathcal{O}_K)$. However, this lattice always contains the distinguished vector $J(1) = (1, 1, \dots, 1)$, which imposes a rigid constraint. To eliminate this obstruction, we consider the trace-zero projection defined by
\[
    \alpha \longmapsto \alpha^\perp := n\alpha - \operatorname{Tr}_{K/\mathbb{Q}}(\alpha).
\]
The image
\[
    \mathcal{O}_K^\perp := \{\alpha^\perp : \alpha \in \mathcal{O}_K\}
\]
is a free $\mathbb{Z}$-module of rank $n-1$. Applying the Minkowski embedding, the corresponding lattice $\Lambda^\perp := J(\mathcal{O}_K^\perp) \subset K_{\mathbb{R}}$ lies entirely in the orthogonal complement of $J(1)$.
\begin{definition}
    Let $K$ be a number field of degree $n$. The \emph{shape of the number field} $K$, denoted $\operatorname{sh}(K)$, is defined as the shape of its $(n-1)$-dimensional projected lattice $\Lambda^\perp$, viewed as an element of the shape space $\mathcal{S}_{n-1}$.
\end{definition}
\noindent To describe this shape concretely, one can choose an integral basis of the form
$\{1, \alpha_1, \dots, \alpha_{n-1}\}$
of $\mathcal{O}_K$. Since the element $1$ projects strictly to the zero vector, the remaining projected elements $
    \{\alpha_1^\perp, \dots, \alpha_{n-1}^\perp\}$
form a complete $\mathbb{Z}$-basis for the projected lattice $\Lambda^\perp$. The associated Gram matrix is then computed as
\[
    \operatorname{Gr}(\Lambda^\perp) = \bigl( \langle \alpha_k^\perp, \alpha_j^\perp \rangle \bigr)_{1 \le k, j \le n-1},
\]
and its equivalence class in
\[
    \mathcal{S}_{n-1} = \operatorname{GL}_{n-1}(\mathbb{Z}) \backslash \mathcal{P}_{n-1} / \mathbb{R}^\times.
\]
\begin{remark}When $K$ is totally real, we identify $K_{\mathbb R}$ with the Euclidean space $\mathbb R^n$ with standard inner product. In this case, 
\[\langle J(x), J(y)\rangle=\sum_{i=1}^n \sigma_i(x) \sigma_i(y)=\op{Tr}_{K/\Q}(xy).\]
\end{remark}

\subsection{Integral bases of totally real multiquadratic number fields}
\par 
In this section, we recall, for the reader’s convenience, the explicit description of an integral basis for multiquadratic number fields from \cite{Man}. These results will be used repeatedly in our subsequent analysis of shapes and their distribution.

\par
Let $n \ge 2$ and let $K_n/\Q$ be a real Galois extension with Galois group
$\op{Gal}(K_n/\Q)\simeq \F_2^n$. 
Such a field is \emph{multi-quadratic}, i.e., may be written in the form
\begin{equation}\label{defn of K} K_n=\Q\bigl(\sqrt{D_{2^0}},\ldots,\sqrt{D_{2^{n-1}}}\bigr),
\end{equation}
where $D_{2^0},\ldots,D_{2^{n-1}}>1$ are squarefree natural numbers. Further, these integers are pairwise distinct and satisfy the nondegeneracy condition that the squarefree part of $D_{2^i}D_{2^j}$ is not equal to $D_{2^k}$ whenever $i,j,k$ are distinct.  
Under this hypothesis the compositum has full degree $[K_n:\Q]=2^n$. For notational convenience we write $\mathcal D := \{D_{2^0},\ldots,D_{2^{n-1}}\}$ which we refer to as a generating set for $K_n$. We further set $\ell := 2^n-1$.

The indexing of the generators by powers of $2$ is chosen so as to represent the full lattice of quadratic subfields of $K_n$.  
Indeed, every nontrivial quadratic subextension corresponds to a nonempty subset of $\mathcal D$, and hence to an integer $1\le j\le \ell$ with binary expansion
\[
j=i+2^k,\qquad 0\le i\le 2^k-1,\; 1\le k\le n-1 .
\]

\begin{lemma}
Let $K_n$ be as in \eqref{defn of K}. Then the quadratic subfields of $K_n$ are precisely the fields $\Q(\sqrt{D_j})$ for $1\le j\le \ell$, where
\begin{equation}\label{eq:quadratic-subfields}
D_j=\frac{D_iD_{2^k}}{\gcd(D_i,D_{2^k})^2}.
\end{equation}
We adopt the convention $D_0=1$.
\end{lemma}
\begin{proof}
 
Since $\op{Gal}(K_n/\Q)\simeq \F_2^n$, the Galois group is an
$n$–dimensional $\mathbb F_2$–vector space.
Quadratic subextensions of $K$ correspond bijectively to
index-$2$ subgroups, equivalently to nonzero linear functionals
on $\op{Gal}(K_n/\Q)$.
Hence $K_n$ has exactly $2^n-1=\ell$ distinct quadratic subfields.

For any nonempty subset $S\subset\{0,\dots,n-1\}$,
set
\[
\alpha_S=\prod_{r\in S}\sqrt{D_{2^r}}.
\]
Distinct subsets $S$ give distinct squarefree integers $D_S$. Thus the associated quadratic fields are distinct. Since there are $\ell$ nonempty subsets, these exhaust all
quadratic subextensions of $K_n$.
\par Writing $j=\sum_{r\in S}2^r$ with
$j=i+2^k$ ($0\le i<2^k$),
we obtain
\[
D_j=\operatorname{sqf}(D_iD_{2^k})
=\frac{D_iD_{2^k}}{\gcd(D_i,D_{2^k})^2},
\]
with the convention $D_0=1$.
\end{proof}
Since each generator $D_i \in \mathcal{D}$ is a squarefree integer, it naturally follows that $D_i \not\equiv 0 \pmod 4$, leaving $D_i \equiv 1, 2,$ or $3 \pmod 4$. To systematically classify all multiquadratic extensions, Chatelain \cite{Chatelain} showed that one can choose a generating set for $K_n$ such that at most one generator is congruent to $2 \pmod 4$ and at most one is congruent to $3 \pmod 4$. 

\begin{lemma}[{\cite[p.~10]{Chatelain}}] \label{lem:generator_congruence}
    Let $K_n = \Q\bigl(\sqrt{D_{2^0}}, \ldots, \sqrt{D_{2^{n-1}}}\bigr)$ be a multiquadratic number field with generating set $\mathcal{D} = \{D_{2^0}, \ldots, D_{2^{n-1}}\}$. Then, without loss of generality, $\mathcal{D}$ contains at most one generator congruent to $2 \pmod 4$ and at most one generator congruent to $3 \pmod 4$.
\end{lemma}

In view of above lemma, the study of multiquadratic extensions can be divided to three distinct categories based on the congruences of their generators modulo $4$. Specifically, up to reordering, the generating set $\mathcal{D}$ falls into exactly one of the following cases:
\begin{enumerate}
    \item $\mathcal{D} \equiv \{1, 1, \ldots, 1, 1, 1\} \pmod 4$.
    \item $\mathcal{D} \equiv \{1, 1, \ldots, 1, 1, 2\} \pmod 4$ \quad or \quad $\mathcal{D} \equiv \{1, 1, \ldots, 1, 1, 3\} \pmod 4$.
    \item $\mathcal{D} \equiv \{1, 1, \ldots, 1, 2, 3\} \pmod 4$.
\end{enumerate}We note that $2$ is unramified in $K_n$ precisely in Case (1). Based on this classification, Chatelain further provides an explicit construction for the integral basis of $K_n$ (cf. \cite{Chatelain}, \cite[Theorem 2.1]{Man}).

\begin{theorem} \label{thm:integral_basis}
    Let $K_n = \Q\bigl(\sqrt{D_{2^0}}, \ldots, \sqrt{D_{2^{n-1}}}\bigr)$ be a multiquadratic number field of degree $2^n$, and let $D_1, \ldots, D_{\ell}$ denote the generators of its $\ell = 2^n - 1$ quadratic subfields. A $\Z$-basis for the ring of integers $\mathcal{O}_{K_n}$ can be constructed according to the congruences of the generating set $\mathcal{D} = \{D_{2^0}, \ldots, D_{2^{n-1}}\}$ modulo $4$:
    \begin{enumerate}
        \item If $\mathcal{D} \equiv \{1, 1, \ldots, 1\} \pmod 4$, then a $\Z$-basis of $\mathcal{O}_{K_n}$ is given by the set of Galois conjugates
        \[
            \bigl\{ \sigma(E) \bigm| \sigma \in \operatorname{Gal}(K_n/\Q) \bigr\},
        \]
        where $E := \frac{1}{2^n}\bigl(1 + \sqrt{D_1} + \cdots + \sqrt{D_{\ell}}\bigr)$.
        
        \item If $\mathcal{D}$ is congruent to either $\{1, \ldots, 1, 2\} \pmod 4$ or $\{1, \ldots, 1, 3\} \pmod 4$, and we assume $D_{2^{n-1}}$ is the unique generator not congruent to $1 \pmod 4$, then a $\Z$-basis of $\mathcal{O}_{K_n}$ is given by
        \[
            \bigl\{ \sigma(E_i) \bigm| \sigma \in \operatorname{Gal}\bigl(K_n/\Q(\sqrt{D_{2^{n-1}}})\bigr), \; i \in \{1, 2\} \bigr\},
        \]
        where the elements $E_1$ and $E_2$ are defined as:
        \begin{align*}
            E_1 &:= \frac{1}{2^{n-1}}\bigl(1 + \sqrt{D_1} + \cdots + \sqrt{D_{2^{n-1}-1}}\bigr), \\
            E_2 &:= \frac{1}{2^{n-1}}\bigl(\sqrt{D_{2^{n-1}}} + \cdots + \sqrt{D_{\ell}}\bigr).
        \end{align*}

        \item If $\mathcal{D} \equiv \{1, \ldots, 1, 2, 3\} \pmod 4$, where $D_{2^{n-2}} \equiv 2 \pmod 4$ and $D_{2^{n-1}} \equiv 3 \pmod 4$, then a $\Z$-basis of $\mathcal{O}_{K_n}$ is given by
        \[
            \bigl\{ \sigma(E_i) \bigm| \sigma \in \operatorname{Gal}\bigl(K_n/\Q(\sqrt{D_{2^{n-2}}}, \sqrt{D_{2^{n-1}}})\bigr), \; i \in \{1, 2, 3, 4\} \bigr\},
        \]
        where the elements $E_1, E_2, E_3,$ and $E_4$ are defined as:
        \begin{align*}
            E_1 &:= \frac{1}{2^{n-2}}\bigl(1 + \sqrt{D_1} + \cdots + \sqrt{D_{2^{n-2}-1}}\bigr), \\
            E_2 &:= \frac{1}{2^{n-2}}\bigl(\sqrt{D_{2^{n-2}}} + \cdots + \sqrt{D_{2^{n-1}-1}}\bigr), \\
            E_3 &:= \frac{1}{2^{n-2}}\bigl(\sqrt{D_{2^{n-1}}} + \cdots + \sqrt{D_{3\cdot 2^{n-2}-1}}\bigr), \\
            E_4 &:= \frac{1}{2^{n-1}}\bigl(\sqrt{D_{2^{n-2}}} + \cdots + \sqrt{D_{2^{n-1}-1}} + \sqrt{D_{3\cdot 2^{n-2}}} + \cdots + \sqrt{D_{\ell}}\bigr).
        \end{align*}
    \end{enumerate}
\end{theorem}
The following corollary provides the discriminant of the multiquadratic field $K_n$.
\begin{corollary}\cite[Corollary 2.2]{Man}\label{disc corollary Man}
    Assume $K_n, \ell, D_i$ with $1\leq i\leq \ell$, and $\mathcal{D}$ be as in the above theorem. Then the discriminant of the multiquadratic field $K_n$ is given by 
    $$ \Delta_{K_n} = (2^r \op{rad}(D_{2^0}D_{2^1}\cdots D_{2^{n-1}}))^{2^{n-1}} = 2^{2^{n-1}r}\prod\limits_{j=1}^{\ell}D_{j},  $$
    where $\op{rad}(k)$ denotes the radical of $k \in \Z$, and 
    \[
    r =
\begin{cases}
0 & \text{if } (D_{2^{n-2}}, D_{2^{n-1}}) \equiv (1,1) \pmod{4}, \\[6pt]
2 & \text{if } (D_{2^{n-2}}, D_{2^{n-1}}) \equiv (1,2) \text{ or } (1,3) \pmod{4}, \\[6pt]
3 & \text{if } (D_{2^{n-2}}, D_{2^{n-1}}) \equiv (2,3) \pmod{4}.
\end{cases}
    \]
\end{corollary}
\section{Explicit Construction of an Integral Basis}
\par In this short section, we formulate a convenient way to represent integral bases in all cases. The relevant notation will be helpful in describing the associated Gram matrices in the next section. 
\subsection*{Case 1: \texorpdfstring{$\mathcal{D}$ congruent to $\{1, \ldots, 1\} \pmod 4$}{Case 1}}
 In this case, there is an elegant way to track the action of the Galois group on the generators of $K_n$. Specifically, $\operatorname{Gal}(K_n/\Q) \cong \F_2^n$ is generated by automorphisms that flip the sign of exactly one fundamental generator $\sqrt{D_{2^k}}$ while fixing the rest. As a result, the signs of all the composite radicals $\sqrt{D_j}$ change accordingly.

Since there are $2^n$ Galois automorphisms acting on the $2^n$ basis elements $\{1, \sqrt{D_1}, \ldots, \sqrt{D_{\ell}}\}$, we can encode these simultaneous sign changes using a $2^n \times 2^n$ matrix with entries in $\{\pm 1\}$. We construct these matrices recursively as follows:

Let $A_0 = [1]$. For $n \ge 0$, we define the sequence of block matrices:
\[ 
    A_{n+1} = \begin{bmatrix}
        A_n & A_n \\
        A_n & -A_n
    \end{bmatrix}. 
\]

For example, $A_1, A_2$ are given as:
\[ 
    A_1 = \begin{bmatrix}
        1 &  1 \\
        1 & -1
    \end{bmatrix}, \qquad 
    A_2 = \begin{bmatrix}
        1 &  1 &  1 &  1 \\
        1 & -1 &  1 & -1 \\
        1 &  1 & -1 & -1 \\
        1 & -1 & -1 &  1
    \end{bmatrix}. 
\]

The explicit connection between the matrix $A_n$ and the Galois embeddings of $K_n$ is formalized by Chatelain:

\begin{lemma}[{\cite[Corollaire 1]{Chatelain}}] \label{lem:galois_matrix}
    Let $A_n = (a_{i,j})$ for $0 \leq i, j \leq \ell$, and let $D_0 = 1$. The embeddings $\sigma_i \in \operatorname{Gal}(K_n/\Q)$ can be ordered such that 
    \[
        a_{i,j} = \frac{\sigma_i(\sqrt{D_j})}{\sqrt{D_j}}.
    \]
\end{lemma}

From its recursive definition, we can see that $A_n$ is symmetric. Furthermore, using induction on $n$, one can easily show that $A_n^2 = 2^n I_{2^n}$, where $I_{2^n}$ denotes the $2^n \times 2^n$ identity matrix. For proving this, suppose it holds for a given $n$, i.e., $A_n^2 = 2^n I_{2^n}$. Then we see that:
\begin{align*}
    A_{n+1}^2 &= \begin{bmatrix}
        A_n & A_n \\
        A_n & -A_n
    \end{bmatrix} 
    \begin{bmatrix}
        A_n & A_n \\
        A_n & -A_n
    \end{bmatrix} \\
    &= \begin{bmatrix}
        2A_n^2 & 0 \\
        0 & 2A_n^2
    \end{bmatrix} \\
    &= \begin{bmatrix}
        2(2^n I_{2^n}) & 0 \\
        0 & 2(2^n I_{2^n})
    \end{bmatrix} \\
    &= 2^{n+1} I_{2^{n+1}},
\end{align*}
which proves the identity. 

Using the above properties of $A_n$, we can rewrite the integral basis for $\mathcal{O}_{K_n}$.

\begin{lemma}[{\cite[Th\'{e}or\`{e}me 1(b)]{Chatelain}}] \label{lem:integral_basis_case1}
    Let $v \in (K_n)^{2^n}$ be the column vector defined by 
    \[ 
        v := \frac{1}{2^n}\bigl(1, \sqrt{D_1}, \ldots, \sqrt{D_{\ell}}\bigr)^T. 
    \]
    If we let $\alpha_j$ denote the $j$-th component of the vector $A_n v$ (for $0 \leq j \leq \ell$), then the set 
    \[ 
        \mathcal{B}_1 = \{\alpha_0, \alpha_1, \ldots, \alpha_{\ell}\} 
    \] 
    forms an integral basis for $\mathcal{O}_{K_n}$. Furthermore, the absolute discriminant of $K_n$ is given by
    \[ 
        \Delta(K_n) = \prod_{j=1}^{\ell} D_j. 
    \]
\end{lemma}

As we are in Case 1, every $D_j \equiv 1 \pmod 4$. Consequently, the discriminant $\Delta(K_n)$ is strictly an odd integer, which provides immediate confirmation that the prime $2$ is unramified in the extension $K_n/\Q$.

To compute the shape of $K_n$, we choose an integral basis that contains the element $1$. The shape is defined by projecting the Minkowski embedding of $\mathcal{O}_{K_n}$ orthogonally to the vector $J(1) = (1, 1, \ldots, 1)$, so having $1$ in our basis  simplifies this projection. We can obtain this by making a unimodular modification to the basis $\mathcal{B}_1$. Recall that $\alpha_j$ is the $j$-th component of the vector $A_n v$. If we sum all the elements of $\mathcal{B}_1$, it follows that
\[
    \sum_{j=0}^{\ell} \alpha_j = 1.
\]
This identity holds because the sum of the entries in the first column of the matrix $A_n$ is $2^n$ (which cancels the factor of $1/2^n$ in $v$), while the sum of the entries in any other column of $A_n$ is exactly $0$. 

Using this, we can write $\alpha_0$ as:
\[
    \alpha_0 = 1 - \sum_{j=1}^{\ell} \alpha_j.
\]
As $\alpha_0$ can be expressed as an integer combination of $1$ and the other $\alpha_j$ elements; replacing $\alpha_0$ with $1$ gives a change of basis with determinant $\pm 1$. Thus, we can define our refined integral basis for $\mathcal{O}_{K_n}$ as:
\[
    \mathcal{B}_1' := \{1, \alpha_1, \ldots, \alpha_{\ell}\}.
\]

\subsection*{Case 2: \texorpdfstring{$\mathcal{D}$ congruent to $\{1, \ldots, 1, 2\}$ or $\{1, \ldots, 1, 3\} \pmod 4$}{Case 2}}

In this case, exactly one fundamental generator is not congruent to $1 \pmod 4$. As established in Theorem \ref{thm:integral_basis}, we assume without loss of generality that this distinguished generator is $D_{2^{n-1}}$. The prime $2$ ramifies in $K_n/\Q$, which changes the structure of the ring of integers. 

Let $m = n - 1$. The integral basis for $\mathcal{O}_{K_n}$ in this case is generated by the action of the subgroup $H = \operatorname{Gal}\bigl(K_n/\Q(\sqrt{D_{2^m}})\bigr)$ on two specific elements, $E_1$ and $E_2$. The subgroup $H$ is isomorphic to $\F_2^m$ and acts exclusively by changing the signs of the fundamental generators $\sqrt{D_{2^0}}, \ldots, \sqrt{D_{2^{m-1}}}$. Let $L = 2^m - 1$. Note that the terms comprising $E_1$ are exactly the square roots of the subfield generators $D_0, \ldots, D_L$ (with our standard convention $D_0 = 1$), while the terms comprising $E_2$ correspond to the remaining half of the subfield generators, which can be indexed as $D_{2^m}, D_{2^m+1}, \ldots, D_{2^m+L}$. Since $H$ acts on the unramified generators exactly as the full Galois group acts in a multiquadratic field of degree $2^m$, we can  track these sign changes using the $2^m \times 2^m$ matrix $A_m$, which was defined recursively in the previous case. 

Let us define two column vectors of length $2^m$ in $(K_n)^{2^m}$:
\begin{align*}
    v^{(1)} &:= \frac{1}{2^m}\bigl(1, \sqrt{D_1}, \ldots, \sqrt{D_L}\bigr)^T, \\
    v^{(2)} &:= \frac{1}{2^m}\bigl(\sqrt{D_{2^m}}, \sqrt{D_{2^m+1}}, \ldots, \sqrt{D_{2^m+L}}\bigr)^T.
\end{align*}

We define the elements $\alpha_j$ and $\beta_j$ (for $0 \leq j \leq L$) as the $j$-th components of the matrix-vector products $A_m v^{(1)}$ and $A_m v^{(2)}$, respectively:
\[
    \begin{bmatrix} \alpha_0 \\ \alpha_1 \\ \vdots \\ \alpha_L \end{bmatrix} = A_m v^{(1)} 
    \qquad \text{and} \qquad 
    \begin{bmatrix} \beta_0 \\ \beta_1 \\ \vdots \\ \beta_L \end{bmatrix} = A_m v^{(2)}.
\]
The sets $\{\alpha_j\}_{j=0}^L$ and $\{\beta_j\}_{j=0}^L$ are precisely the sets of Galois conjugates of $E_1$ and $E_2$ under $H$. Therefore, an  integral basis for $K_n$ is given by the union:
\[
    \mathcal{B}_2 = \{\alpha_0, \ldots, \alpha_L\} \cup \{\beta_0, \ldots, \beta_L\}.
\]

\subsubsection*{Refinement of the Basis}

Just as in Case 1, computing the shape of the Minkowski lattice $J(\mathcal{O}_{K_n})$ requires projecting orthogonally onto the subspace $J(1)^\perp$. To make this projection tractable, we modify our basis so that $1$ is explicitly included. 

We observe that the elements $\alpha_j$ behave identically to the integral basis elements of an unramified degree $2^m$ extension. Since the first column of the matrix $A_m$ consists entirely of $1$s, and the remaining columns sum to zero, summing the components of $A_m v^{(1)}$ yields:
\[
    \sum_{j=0}^L \alpha_j = 1.
\]
So, we see that:
\[
    \alpha_0 = 1 - \sum_{j=1}^L \alpha_j.
\]
As $\alpha_0$ can be written as an integer linear combination of $1$ and the elements $\alpha_1, \ldots, \alpha_L$, replacing $\alpha_0$ with $1$ is a unimodular transformation (a change of basis with determinant $\pm 1$). 

Therefore, we can write our refined integral basis for $\mathcal{O}_{K_n}$ in Case 2 as:
\[
    \mathcal{B}_2' := \{1, \alpha_1, \ldots, \alpha_L, \beta_0, \beta_1, \ldots, \beta_L\}.
\]

\subsection*{Case 3: \texorpdfstring{$\mathcal{D} \equiv \{1, \ldots, 1, 2, 3\} \pmod 4$}{Case 3}}

In this final case, exactly two fundamental generators are not congruent to $1 \pmod 4$. As  in Theorem \ref{thm:integral_basis}, we assume without loss of generality that these are $D_{2^{n-2}} \equiv 2 \pmod 4$ and $D_{2^{n-1}} \equiv 3 \pmod 4$. Let $m:= n - 2$. The integral basis is generated by the action of the subgroup 
\[
    H = \operatorname{Gal}\bigl(K_n/\Q(\sqrt{D_{2^m}}, \sqrt{D_{2^{m+1}}})\bigr)
\]
on four specific elements: $E_1, E_2, E_3,$ and $E_4$. The subgroup $H \cong \F_2^m$ has order $2^m$ and acts  by changing the signs of the $m$ fundamental generators $\sqrt{D_{2^0}}, \ldots, \sqrt{D_{2^{m-1}}}$. 
Letting $L = 2^m - 1$, we define four column vectors in $(K_n)^{2^m}$:
\begin{align*}
    v^{(1)} &:= \frac{1}{2^m}\bigl(1, \sqrt{D_1}, \ldots, \sqrt{D_L}\bigr)^T, \\
    v^{(2)} &:= \frac{1}{2^m}\bigl(\sqrt{D_{2^m}}, \sqrt{D_{2^m+1}}, \ldots, \sqrt{D_{2^m+L}}\bigr)^T, \\
    v^{(3)} &:= \frac{1}{2^m}\bigl(\sqrt{D_{2^{m+1}}}, \sqrt{D_{2^{m+1}+1}}, \ldots, \sqrt{D_{2^{m+1}+L}}\bigr)^T, \\
    v^{(4)} &:= \frac{1}{2^m}\bigl(\sqrt{D_{3\cdot 2^m}}, \sqrt{D_{3\cdot 2^m+1}}, \ldots, \sqrt{D_{3\cdot 2^m+L}}\bigr)^T.
\end{align*}

As $H$ acts on these disjoint blocks exactly as the full Galois group acts on a purely unramified multiquadratic field of degree $2^m$, we can   track the simultaneous sign changes using the $2^m \times 2^m$ symmetric transition matrix $A_m$.

We define the intermediate sets of elements $\alpha_j, \beta_j, \gamma_j,$ and $\eta_j$ (for $0 \leq j \leq L$) as the components of the matrix-vector products:
\[
    \begin{bmatrix} \alpha_0 \\ \vdots \\ \alpha_L \end{bmatrix} = A_m v^{(1)}, \quad
    \begin{bmatrix} \beta_0 \\ \vdots \\ \beta_L \end{bmatrix} = A_m v^{(2)}, \quad
    \begin{bmatrix} \gamma_0 \\ \vdots \\ \gamma_L \end{bmatrix} = A_m v^{(3)}, \quad
    \begin{bmatrix} \eta_0 \\ \vdots \\ \eta_L \end{bmatrix} = A_m v^{(4)}.
\]

By directly comparing these definitions to the elements provided in Theorem \ref{thm:integral_basis}, we see that the Galois conjugates of $E_1, E_2,$ and $E_3$ are exactly the elements $\alpha_j, \beta_j,$ and $\gamma_j$, respectively. 

However, the element $E_4$ has a denominator of $2^{n-1} = 2^{m+1}$ and its sum spans both the second and fourth blocks of radicals. Notice that the elements $\beta_0$ and $\eta_0$ (which correspond to the first row of $A_m$) are exactly the sums of the entries in $v^{(2)}$ and $v^{(4)}$, respectively. Therefore, $E_4$ can be expressed algebraically as:$$    E_4 = \frac{1}{2} (\beta_0 + \eta_0).$$Consequently, applying the subgroup $H$ to generate the Galois conjugates of $E_4$ simply yields the averages of our intermediate conjugate pairs. We define these conjugates as $\delta_j$:
\[
    \delta_j := \frac{1}{2}(\beta_j + \eta_j) \quad \text{for } 0 \leq j \leq L.
\]
Therefore, an  integral basis for $\mathcal{O}_{K_n}$ in Case 3 is given by the union of these four sets:
\[
    \mathcal{B}_3 = \{\alpha_0, \ldots, \alpha_L\} \cup \{\beta_0, \ldots, \beta_L\} \cup \{\gamma_0, \ldots, \gamma_L\} \cup \{\delta_0, \ldots, \delta_L\}.
\]

\subsubsection*{Refinement of the Basis}

As we know that computing the shape of the Minkowski lattice requires projecting the lattice orthogonally onto $J(1)^\perp$. For this, our basis should contain $1$. 

The elements $\alpha_j$ are constructed from the block $v^{(1)}$, which is identical in form to the unramified basis structure of degree $2^m$. Since the first column of $A_m$ is entirely $1$s and the remaining columns sum to zero, summing the components of $A_m v^{(1)}$ yields:
\[
    \sum_{j=0}^L \alpha_j = 1.
\]
So, we have
\[
    \alpha_0 = 1 - \sum_{j=1}^L \alpha_j.
\]
As $\alpha_0$ is an integer linear combination of $1$ and the elements $\alpha_1, \ldots, \alpha_L$, replacing $\alpha_0$ with $1$ is a strictly unimodular transformation (a change of basis with determinant $\pm 1$). 

Thus, we can write our refined integral basis for $\mathcal{O}_{K_n}$ in Case 3 as:
\[
    \mathcal{B}_3' := \{1, \alpha_1, \ldots, \alpha_L, \beta_0, \ldots, \beta_L, \gamma_0, \ldots, \gamma_L, \delta_0, \ldots, \delta_L\}.
\]
\section{The Gram Matrices and Shapes of $K_n$}
\par In this section, we describe the shapes of totally real multiquadratic number fields.
\subsection{Case 1: \texorpdfstring{$\mathcal{D} \equiv \{1, 1, \ldots, 1\} \pmod 4$}{Case 1}} 
\subsubsection*{The Gram Matrix} Recall that $\mathcal{B}_1' := \{1, \alpha_1, \ldots, \alpha_{\ell}\}$ is the modified integral basis, where $\alpha_i$'s are as in Lemma \ref{lem:integral_basis_case1}. For computing the shape of $K_n$, we first compute the Gram matrix of the lattice, which requires us to determine the inner products of our basis elements under the canonical Minkowski embedding $J : K_n \hookrightarrow \R^{2^n}$. Let us formally define the embedded vectors:
\[
    v_j := J(\sqrt{D_j}) \quad \text{for } 0 \leq j \leq \ell,
\]
where we adopt the convention that $D_0 = 1$ (so that $v_0 = J(1)$).

\begin{lemma} \label{lem:orthogonal_radicals}
    With the definitions above, the vectors $v_j$ are mutually orthogonal in $\R^{2^n}$. Moreover, their squared Euclidean norms are given by
    \[
        |v_j|^2 = \langle v_j, v_j \rangle = 2^n D_j.
    \]
\end{lemma}

\begin{proof}
    By the definition of the Minkowski inner product, we see that:
    \[
        |v_j|^2 = \langle v_j, v_j \rangle = \operatorname{Tr}_{K_n/\Q}(D_j) = \sum_{k=0}^{\ell} D_j = 2^n D_j.
    \]
    Now suppose that $i \neq j$. The product $\sqrt{D_i}\sqrt{D_j}$ can be written as $c \sqrt{D_k}$ for some integer $c = \gcd(D_i, D_j)$ and some subfield index $k \ge 1$. It follows that:
    \begin{align*}
        \langle v_i, v_j \rangle &= \operatorname{Tr}_{K_n/\Q}(\sqrt{D_i}\sqrt{D_j}) \\
        &= c \operatorname{Tr}_{K_n/\Q}(\sqrt{D_k}) \\
        &= c \sqrt{D_k} \sum_{m=0}^{\ell} a_{m,k} = 0,
    \end{align*}
    where the final equality holds because the sum of the Galois signs in any column $k \ge 1$ of the matrix $A_n$ is zero.
\end{proof}

Since the generators $D_j$ are squarefree integers, $D_j \ge 2$ for all $1 \leq j \leq \ell$. Recalling our convention that $D_0 = 1$, the lemma immediately implies that:
\[ 
    |v_0| \leq |v_j| \quad \text{for } 1 \leq j \leq \ell.
\]

Let $\mathcal{B}_v = \{v_0, v_1, \ldots, v_\ell\}$ denote the orthogonal $\mathbb{Q}$-basis of $K_n$. The Gram matrix with respect to this basis, denoted $G_v$, is simply the diagonal matrix of the squared norms. By Lemma \ref{lem:orthogonal_radicals}, we have:
$$    G_v = \begin{bmatrix}
        |v_0|^2 & 0 & \cdots & 0 \\
        0 & |v_1|^2 & \cdots & 0 \\
        \vdots & \vdots & \ddots & \vdots \\
        0 & 0 & \cdots & |v_{\ell}|^2
    \end{bmatrix} 
    = 2^n \begin{bmatrix}
        1 & 0 & \cdots & 0 \\
        0 & D_1 & \cdots & 0 \\
        \vdots & \vdots & \ddots & \vdots \\
        0 & 0 & \cdots & D_{\ell}
    \end{bmatrix}.$$

\vspace{0.3cm}

Recall that the integral basis $\mathcal{B}_1$ is defined via the matrix equation:$$    \begin{bmatrix}
        \alpha_0 \\ \alpha_1 \\ \vdots \\ \alpha_{\ell}
    \end{bmatrix} 
    = \frac{1}{2^n} A_n \begin{bmatrix}
        v_0 \\ v_1 \\ \vdots \\ v_{\ell}
    \end{bmatrix}.$$Therefore, the transition matrix from the orthogonal basis $\mathcal{B}_v$ to the integral basis $\mathcal{B}_1$ is given by $\frac{1}{2^n} A_n$. Consequently, the Gram matrix with respect to $\mathcal{B}_1$ is obtained via the congruent transformation:$$    G_{\mathcal{B}_1} = \left(\frac{1}{2^n} A_n\right) G_v \left(\frac{1}{2^n} A_n\right)^T.$$Substituting our diagonal expression for $G_v$, and using the fact that $A_n$ is a symmetric matrix, the precise formulation for the Gram matrix of $\mathcal{O}_{K_n}$ is given by:
\[ 
    G_{\mathcal{B}_1} = \frac{1}{2^n} A_n \begin{bmatrix}
        1 & 0 & \cdots & 0 \\
        0 & D_1 & \cdots & 0 \\
        \vdots & \vdots & \ddots & \vdots \\
        0 & 0 & \cdots & D_{\ell}
    \end{bmatrix} A_n.
\]

\subsubsection*{Shape of \texorpdfstring{$K_n$}{Kn}}

With the Gram matrix of $\mathcal{O}_{K_n}$ (explicitly determined above), we are now in a position to compute the shape of the multiquadratic field $K_n$. Recall that the shape of a number field, denoted $\operatorname{sh}(K_n)$, is formally defined as the similarity class (up to rotation and uniform scaling) of the lattice obtained by projecting the Minkowski embedding of its ring of integers orthogonally to the one-dimensional subspace spanned by $v_0 = J(1)$. 

Since our modified integral basis $\mathcal{B}_1' = \{1, \alpha_1, \ldots, \alpha_{\ell}\}$ explicitly contains $1$, this projection becomes highly tractable. Let $\Lambda$ denote the lattice $J(\mathcal{O}_{K_n})$ in $\R^{2^n}$, and let $\Lambda^\perp$ be its orthogonal projection onto $v_0^\perp$. The lattice $\Lambda^\perp$ is generated by the following basis elements:
\[
    \mathcal{B}^\perp = \{\alpha_1^\perp, \ldots, \alpha_{\ell}^\perp\},
\]
where $\alpha_j^\perp = \alpha_j - \frac{\langle \alpha_j, v_0 \rangle}{|v_0|^2} v_0$ represents the orthogonal projection of $\alpha_j$. 

To compute $\alpha_j^\perp$ explicitly, we recall the definition of $\alpha_j$ from the transition matrix $A_n = (a_{i,k})$. As the first column of $A_n$ consists entirely of $1$s, we have $a_{j,0} = 1$ for all $j$. Thus, expanding $\alpha_j$ yields:
\[
    \alpha_j = \frac{1}{2^n} \sum_{k=0}^{\ell} a_{j,k} v_k = \frac{1}{2^n} v_0 + \frac{1}{2^n} \sum_{k=1}^{\ell} a_{j,k} v_k.
\]
Since the underlying vectors $v_1, \ldots, v_{\ell}$ are orthogonal to $v_0$ (by Lemma \ref{lem:orthogonal_radicals}), it follows that:
\[
    \langle \alpha_j, v_0 \rangle = \left\langle \frac{1}{2^n} v_0, v_0 \right\rangle = \frac{1}{2^n} |v_0|^2.
\]
Hence, we have
\[
    \alpha_j^\perp = \frac{1}{2^n} \sum_{k=1}^{\ell} a_{j,k} v_k \quad \text{for } 1 \leq j \leq \ell.
\]

Thus, we see that the projected basis elements are purely linear combinations of the mutually orthogonal vectors $v_1, \ldots, v_{\ell}$ with coefficients $\pm 1/2^n$. We can now  compute the entries of the projected Gram matrix $G_{\Lambda^\perp} = [\langle \alpha_i^\perp, \alpha_j^\perp \rangle]$. Using the orthogonality of the $v_k$, we obtain:
\begin{align*}
    \langle \alpha_i^\perp, \alpha_j^\perp \rangle &= \frac{1}{2^{2n}} \sum_{k=1}^{\ell} a_{i,k} a_{j,k} |v_k|^2 \\
    &= \frac{1}{2^{2n}} \sum_{k=1}^{\ell} a_{i,k} a_{j,k} (2^n D_k) \\
    &= \frac{1}{2^n} \sum_{k=1}^{\ell} a_{i,k} a_{j,k} D_k.
\end{align*}

Let $\tilde{A}_n$ be the $\ell \times \ell$ submatrix obtained by deleting the first row and first column of $A_n$. The above calculation shows that the Gram matrix of $\Lambda^\perp$ is given by:
\[
    G_{\Lambda^\perp} = \frac{1}{2^n} \tilde{A}_n \operatorname{diag}(D_1, \ldots, D_{\ell}) \tilde{A}_n^T.
\]
Normalizing the lattice by $1/|v_1|^2$ naturally introduces the continuous parameters:
\[
    \lambda_j := \frac{|v_j|^2}{|v_1|^2} = \frac{D_j}{D_1} \quad \text{for } 2 \leq j \leq \ell.
\]
The shape of the multiquadratic field $K_n$ is entirely parameterized by these ratios. Following the framework for primitive orthorhombic lattices, the shape of $K_n$ can be concisely represented as:
\[
    \operatorname{sh}(K_n) = \operatorname{sh}(\lambda_2, \ldots, \lambda_{\ell}).
\]
Note that we may always choose a permutation of $D_1, \dots, D_\ell$ such that 
\[1\leq \lambda_{2}\leq \lambda_3\leq  \ldots \leq \lambda_{\ell},\] and the shape $\operatorname{sh}(K_n)$ remains unchanged.
\subsection{Case 2: \texorpdfstring{$\mathcal{D}$ congruent to $\{1, \ldots, 1, 2\}$ or $\{1, \ldots, 1, 3\} \pmod 4$}{Case 2}}
Although we shall not in this article treat the distribution of shapes in Cases 2 and 3, we describe the Gram matrices.
\subsubsection*{The Gram Matrix}

To compute the Gram matrix of $\mathcal{O}_{K_n}$ in this case, we first determine the inner products of the basis $\mathcal{B}_2 = \{\alpha_0, \ldots, \alpha_L\} \cup \{\beta_0, \ldots, \beta_L\}$. 

Recall from Lemma \ref{lem:orthogonal_radicals} that the underlying embedded vectors $v_k = J(\sqrt{D_k})$ are mutually orthogonal in Minkowski space, and $|v_k|^2 = 2^n D_k$. Since our basis elements are partitioned into the unramified block (the $\alpha_j$ constructed from $v_0, \ldots, v_L$) and the ramified block (the $\beta_j$ constructed from $v_{2^m}, \ldots, v_{2^m+L}$), and as these two sets of vectors are disjoint, it immediately follows that every $\alpha_i$ is  orthogonal to every $\beta_j$:
\[
    \langle \alpha_i, \beta_j \rangle = 0 \quad \text{for all } 0 \leq i, j \leq L.
\]
Consequently, the Gram matrix $G_{\mathcal{B}_2}$  splits into a block-diagonal matrix:
\[
    G_{\mathcal{B}_2} = \begin{bmatrix} G_\alpha & 0 \\ 0 & G_\beta \end{bmatrix}.
\]

Let us compute the entries of the $G_\alpha$ block. Using the transition matrix $A_m = (a_{i,k})$, we have $\alpha_i = \frac{1}{2^m} \sum_{k=0}^L a_{i,k} v_k$. Therefore, we see that:
\begin{align*}
    \langle \alpha_i, \alpha_j \rangle &= \frac{1}{2^{2m}} \sum_{k=0}^L a_{i,k} a_{j,k} |v_k|^2 \\
    &= \frac{1}{2^{2m}} \sum_{k=0}^L a_{i,k} a_{j,k} (2^n D_k).
\end{align*}
Since $m = n - 1$, we have $2m = 2n - 2$. So, it follows that $\frac{2^n}{2^{2n-2}} = \frac{1}{2^{n-2}}$. Thus, the Gram matrix for the unramified block is given by:
\[
    G_\alpha = \frac{1}{2^{n-2}} A_m \operatorname{diag}(1, D_1, \ldots, D_L) A_m^T.
\]
By a similar calculation over the ramified indices, the Gram matrix for the $\beta$ block is:
\[
    G_\beta = \frac{1}{2^{n-2}} A_m \operatorname{diag}(D_{2^m}, D_{2^m+1}, \ldots, D_{2^m+L}) A_m^T.
\]

\subsubsection*{Shape of \texorpdfstring{$K_n$}{Kn}}

We now project the Minkowski embedding of $\mathcal{O}_{K_n}$ orthogonally onto the subspace $v_0^\perp$ to determine its shape. Using our refined integral basis $\mathcal{B}_2' = \{1, \alpha_1, \ldots, \alpha_L, \beta_0, \ldots, \beta_L\}$, the projected lattice $\Lambda^\perp$ is generated by the basis:
\[
    \mathcal{B}^\perp = \{\alpha_1^\perp, \ldots, \alpha_L^\perp\} \cup \{\beta_0^\perp, \ldots, \beta_L^\perp\}.
\]

Since the elements $\alpha_j$ identically mirror the structure of Case 1 (but in dimension $2^m$), the projection simply removes their $v_0$ component. Precisely, it follows that:
\[
    \alpha_j^\perp = \alpha_j - \frac{\langle \alpha_j, v_0 \rangle}{|v_0|^2} v_0 = \frac{1}{2^m} \sum_{k=1}^L a_{j,k} v_k \quad \text{for } 1 \leq j \leq L.
\]
Conversely, consider the elements $\beta_j$. Since each is a linear combination of the vectors $v_{2^m}, \ldots, v_{2^m+L}$, all of which are orthogonal to $v_0$, the inner product $\langle \beta_j, v_0 \rangle$ is identically zero. Thus, the elements $\beta_j$ are completely unaffected by the orthogonal projection, and hence we have:
\[
    \beta_j^\perp = \beta_j \quad \text{for } 0 \leq j \leq L.
\]

This confirms that the total dimension of the projected lattice is exactly $L + (L + 1) = 2L + 1 = 2(2^m - 1) + 1 = 2^{m+1} - 1 = 2^n - 1 = \ell$.

The Gram matrix $G_{\Lambda^\perp}$ of the projected lattice retains its exact block-diagonal structure. Let $\tilde{A}_m$ be the $L \times L$ submatrix obtained by deleting the first row and first column of $A_m$. The full projected Gram matrix is:
\[
    G_{\Lambda^\perp} = \frac{1}{2^{n-2}} 
    \begin{bmatrix}
        \tilde{A}_m \operatorname{diag}(D_1, \ldots, D_L) \tilde{A}_m^T & 0 \\
        0 & A_m \operatorname{diag}(D_{2^m}, \ldots, D_{2^m+L}) A_m^T
    \end{bmatrix}.
\]

To parameterize this geometry, we normalize the projected lattice by $1/|v_1|^2$ to extract the continuous parameters:   $ \lambda_j := \frac{D_j}{D_1} \quad \text{for } 2 \leq j \leq \ell.$
Consequently, the shape of $K_n$ in Case 2 is  determined by the parameters $\lambda_2, \ldots, \lambda_\ell$ and once again is represented by $\op{sh}(\lambda_2, \dots, \lambda_\ell)$.

\subsection{Case 3: \texorpdfstring{$\mathcal{D} \equiv \{1, \ldots, 1, 2, 3\} \pmod 4$}{Case 3}}

\subsubsection*{The Gram Matrix}

To compute the Gram matrix $G_{\mathcal{B}_3}$ of the ring of integers $\mathcal{O}_{K_n}$, we first determine the inner products of the refined basis $\mathcal{B}_3' = \{1, \alpha_1, \ldots, \alpha_L, \beta_0, \ldots, \beta_L, \gamma_0, \ldots, \gamma_L, \delta_0, \ldots, \delta_L\}$. Recall that $v_k = J(\sqrt{D_k})$ are mutually orthogonal in Minkowski space with  $|v_k|^2 = 2^n D_k$. Our basis is constructed from four mutually disjoint blocks of these vectors ($v^{(1)}, v^{(2)}, v^{(3)}, v^{(4)}$). Consequently, any basis element constructed exclusively from one block is perfectly orthogonal to an element constructed from a different block. Specifically, the unramified $\alpha$ elements and the $\gamma$ elements are orthogonal to everything outside their respective blocks. The Gram matrix $G_{\mathcal{B}_3}$ decomposes into a block-diagonal structure containing two independent blocks and one coupled block:
\[
    G_{\mathcal{B}_3} = 
    \begin{bmatrix} 
        G_\alpha & 0 & 0 \\ 
        0 & G_\gamma & 0 \\ 
        0 & 0 & G_{\beta, \delta} 
    \end{bmatrix}.
\]

 Using the transition matrix $A_m = (a_{i,k})$, we now determine the inner product for elements within the same block (e.g., the $\alpha$ block):
\begin{align*}
    \langle \alpha_i, \alpha_j \rangle &= \frac{1}{2^{2m}} \sum_{k=0}^L a_{i,k} a_{j,k} |v_k|^2 \\
    &= \frac{1}{2^{2m}} \sum_{k=0}^L a_{i,k} a_{j,k} (2^n D_k).
\end{align*}
Since $m = n - 2$ in this case, the denominator is $2^{2m} = 2^{2n - 4}$. Therefore, we have $\frac{2^n}{2^{2n-4}} = \frac{1}{2^{n-4}}$. 

Using this scaling factor ($\frac{1}{2^{n-4}}$), and letting $V_1, V_2, V_3, V_4$ denote the $(L+1) \times (L+1)$ diagonal matrices containing the subfield generators $D_k$ for the first, second, third, and fourth blocks respectively, the isolated Gram matrices are:
\begin{align*}
    G_\alpha &= \frac{1}{2^{n-4}} A_m V_1 A_m^T, \\
    G_\gamma &= \frac{1}{2^{n-4}} A_m V_3 A_m^T.
\end{align*}

For the coupled block $G_{\beta, \delta}$, we compute the cross-terms. Since $\beta_j$ strictly belongs to the second block and $\eta_j$ strictly belongs to the fourth block, $\langle \beta_i, \eta_j \rangle = 0$. Thus, we see that:
\begin{align*}
    \langle \beta_i, \delta_j \rangle &= \frac{1}{2} \langle \beta_i, \beta_j + \eta_j \rangle = \frac{1}{2} \langle \beta_i, \beta_j \rangle, \\
    \langle \delta_i, \delta_j \rangle &= \frac{1}{4} \langle \beta_i + \eta_i, \beta_j + \eta_j \rangle = \frac{1}{4} \bigl( \langle \beta_i, \beta_j \rangle + \langle \eta_i, \eta_j \rangle \bigr).
\end{align*}
Factoring out the shared transition matrix $A_m$, we can express the $2(L+1) \times 2(L+1)$ coupled Gram block as:
\[
    G_{\beta, \delta} = \frac{1}{2^{n-4}} 
    \begin{bmatrix} A_m & 0 \\ 0 & A_m \end{bmatrix} 
    \begin{bmatrix} V_2 & \frac{1}{2} V_2 \\ \frac{1}{2} V_2 & \frac{1}{4}(V_2 + V_4) \end{bmatrix} 
    \begin{bmatrix} A_m^T & 0 \\ 0 & A_m^T \end{bmatrix}.
\]

\subsubsection*{Shape of \texorpdfstring{$K_n$}{Kn}}

To find the shape, we project the Minkowski lattice orthogonally onto $v_0^\perp$. Using our refined basis $\mathcal{B}_3'$, the projected lattice $\Lambda^\perp$ is generated by:
\[
    \mathcal{B}^\perp = \{\alpha_1^\perp, \ldots, \alpha_L^\perp\} \cup \{\beta_0^\perp, \ldots, \beta_L^\perp\} \cup \{\gamma_0^\perp, \ldots, \gamma_L^\perp\} \cup \{\delta_0^\perp, \ldots, \delta_L^\perp\}.
\]

Since $v_0$ lies strictly within the first block $v^{(1)}$, the orthogonal projection completely ignores the $\beta, \gamma,$ and $\delta$ elements. Just as in the previous cases:
\begin{align*}
    \alpha_j^\perp &= \frac{1}{2^m} \sum_{k=1}^L a_{j,k} v_k \quad \text{for } 1 \leq j \leq L, \\
    \beta_j^\perp &= \beta_j, \quad \gamma_j^\perp = \gamma_j, \quad \delta_j^\perp = \delta_j.
\end{align*}

The total dimension of this projected lattice is $L + 3(L+1) = 4L + 3 = 4(2^m - 1) + 3 = 2^{m+2} - 1 = 2^n - 1 = \ell$.

The Gram matrix $G_{\Lambda^\perp}$ of the projected lattice keeps the identical block structure of $G_{\mathcal{B}_3}$, with the one exception that the unramified block $G_\alpha$ is reduced to the $L \times L$ submatrix $\tilde{A}_m \tilde{V}_1 \tilde{A}_m^T$ (where the $v_0$ component is removed).

Geometrically, this block-diagonal structure proves that $\Lambda^\perp$ is the orthogonal direct sum of several distinct sublattices. The $\alpha^\perp$ and $\gamma^\perp$ blocks correspond directly to primitive orthorhombic lattices. However, the inner matrix of the coupled block $G_{\beta, \delta}$ shows a structural shift. For each internal index $k$, it contains a $2 \times 2$ submatrix of the form:
\[
    \begin{bmatrix} D_{2^m+k} & \frac{1}{2} D_{2^m+k} \\ \frac{1}{2} D_{2^m+k} & \frac{1}{4}(D_{2^m+k} + D_{3\cdot 2^m+k}) \end{bmatrix}.
\]
This strictly corresponds to a lattice generated by the vectors $x = e_1$ and $y = \frac{1}{2}(e_1 + e_2)$ with $e_1, e_2$ orthogonal, which is classically recognized as a base-centered orthorhombic lattice. 
Normalizing the lattice by $1/|v_1|^2$ yields the exact same continuous parameters:   $ \lambda_j := \frac{D_j}{D_1} \quad \text{for } 2 \leq j \leq \ell.$
Thus, regardless of the ramification profile, the final shape of the field $K_n$ is uniquely determined by $\lambda_2, \ldots, \lambda_\ell$.

\section{Density results}
\par In this section, we prove our main density results for totally real multiquadratic number fields $K_n$ of degree $2^n$ in which $2$ is unramified (i.e. in Case 1 as described in the previous section).
\subsection{Parameterization of totally real multiquadratic extensions}
\par Let $n\ge 3$ and put $\ell:=2^n-1$. In this section we prove our density results for the variation of shape parameters, and we begin by parameterizing totally real multiquadratic fields $K_n$ with $\op{Gal}(K_n/\Q)\simeq \F_2^n$ by certain $\ell$-tuples of positive integers $(g_1, \dots, g_\ell)$. Set $a_i:=D_{2^{i-1}}$ for
$1\le i\le n$. Let $G:=\F_2^n$ and $\Phi_n:=G\setminus\{0\}$. We define an
ordering $G=\{\mathbf v_0,\dots,\mathbf v_\ell\}$ as follows. For $j=0, \dots, \ell$ if $j=\sum_{t=1}^n b_t 2^{t-1}$ is the binary expansion of the integer $j$,
we have $\mathbf v_j=(b_1,\dots,b_n)\in\F_2^n$.
In particular, if $\mathbf e_1,\dots,\mathbf e_n$ denotes the standard basis of
$\F_2^n$, then $\mathbf v_{2^{i-1}}=\mathbf e_i$. Note that $\mathbf v_0$ is the $0$ vector. Given $a\in \F_2$, let $\widetilde{a}\in \Z$ be $0$ if $a=0\in \F_2$ and $1$ if $a=1\in \F_2$. 

\begin{definition}
An $\ell$–tuple $(g_1,\dots,g_\ell)$ of positive integers is
\emph{strongly carefree} if the $g_i$ are squarefree and pairwise
coprime and is said to be \emph{nondegenerate} if each coordinate $g_i>1$.
\end{definition}
\noindent We write $\mathbf v_i\cdot \mathbf v_j$ for the standard dot product over $\F_2$.
Given a strongly carefree tuple $(g_1,\dots,g_\ell)$, define for $j=1, \dots, \ell$, 
\begin{equation}\label{eq:D-param-general}
D_j :=
\prod_{\substack{1\le i\le \ell\\ \mathbf v_i\cdot \mathbf v_j=1}}
g_i=\prod_{i=1}^\ell g_i^{\widetilde{\mathbf v_i\cdot \mathbf v_j}}.
\end{equation}
Here, the understanding is that $g_i^0=1$ and $g_i^1=g_i$ where $0$ and $1$ are the trivial and nontrivial elements in $\F_2$ respectively. In particular,
\begin{equation}\label{eq2:D-param-general}
a_j=D_{2^{j-1}}
=\prod_{\substack{1\le i\le \ell\\ \mathbf v_i\cdot \mathbf e_j=1}}
g_i,
\end{equation}
and, by pairwise coprimality,
\[
g_i=\gcd\bigl\{\,D_j:\ \mathbf v_i\cdot \mathbf v_j=1\,\bigr\}.
\]

\begin{proposition}
Let $(g_1,\dots,g_\ell)$ be a nondegenerate strongly carefree tuple.
Then the classes of $a_1,\dots,a_n$ are linearly independent
in $\Q^\times/\Q^{\times 2}$. Thus the number field 
\[K_n=\Q(\sqrt{a_1}, \dots, \sqrt{a_n})\]
is an extension of degree $2^n$ with Galois group $\op{Gal}(K_n/\Q)$ isomorphic to $\F_2^n$. 
\end{proposition}

\begin{proof}
Suppose that
\[
\prod_{j\in J} a_j \in \Q^{\times 2}
\]
for some subset $J\subseteq\{1,\dots,n\}$.
Expanding via \eqref{eq2:D-param-general}, we obtain
\[
\prod_{j\in J} a_j
=
\prod_{i=1}^\ell
g_i^{\sum_{j\in J} \widetilde{\mathbf v_i\cdot \mathbf e_j}}.
\]
Because the $g_i$ are pairwise coprime and squarefree,
this product is a square in $\Q^\times$ if and only if $\sum_{j\in J} \widetilde{\mathbf v_i\cdot \mathbf e_j}\equiv 0 \pmod 2$ for every $1\le i\le\ell$. Equivalently,
\[
\mathbf v_i\cdot\Bigl(\sum_{j\in J} \mathbf e_j\Bigr)=0
\quad\text{for all }i.
\]
Since $\{\mathbf v_1,\dots,\mathbf v_\ell\}$ is the set of all nonzero vectors
in $\F_2^n$, the only vector orthogonal to every $\mathbf v_i$
is the zero vector.  Hence
\[
\sum_{j\in J} \mathbf e_j=0
\quad\text{in }\F_2^n,
\]
so this forces $J=\varnothing$.  Therefore the $a_j$ are linearly independent in
$\Q^\times/\Q^{\times 2}$.
\end{proof}
Each $g_i$ occurs in exactly $2^{n-1}$ of the $D_j$. It follows that
\[
\prod_{j=1}^{\ell} D_j
=
\Bigl(\prod_{i=1}^{\ell} g_i\Bigr)^{2^{\,n-1}}.
\]
\noindent Let $p_1,\dots,p_s$ denote the distinct primes dividing
$a_1\cdots a_n$.  Since the $g_i$ are squarefree and pairwise
coprime, we have
\[
\prod_{i=1}^{\ell} g_i
=
p_1\cdots p_s.
\]
Consequently,
\[
\prod_{j=1}^{\ell} D_j
=
(p_1\cdots p_s)^{2^{\,n-1}}.
\]
According to Lemma \ref{lem:integral_basis_case1}, the discriminant of $K_n$ is of the form
\[
\Delta_{K_n}
=
\prod_{j=1}^\ell D_j.
\]
Recall that
\[
\operatorname{sh}(K_n)
=
\operatorname{sh}\!\left(
\frac{D_2}{D_1},\dots,\frac{D_\ell}{D_1}
\right).
\]
From \eqref{eq:D-param-general} we have
\[
\frac{D_j}{D_1}
=
\frac{\prod_{i=1}^{\ell}
g_i^{\,\widetilde{\mathbf v_i\cdot\mathbf v_j}}}{\prod_{i=1}^{\ell}
g_i^{\,\widetilde{\mathbf v_i\cdot\mathbf v_1}}}=
\frac{
\prod_{\substack{i\\
\mathbf v_i\cdot \mathbf v_j=1\\
\mathbf v_i\cdot \mathbf v_1=0}}
g_i}
{
\prod_{\substack{i\\
\mathbf v_i\cdot \mathbf v_j=0\\
\mathbf v_i\cdot \mathbf v_1=1}}
g_i}.
\]
Both numerator and denominator are monomials of degree $2^{n-2}$.
In particular each ratio is homogeneous of total degree $0$. 
\par There exists a unique permutation $\sigma$ of $\{1,\dots,\ell\}$
such that
\[
D_{\sigma(1)}
\leq 
D_{\sigma(2)}
\leq 
\dots
\leq 
D_{\sigma(\ell)}.
\]
Setting $\lambda_j:=\frac{D_{\sigma(j)}}{D_{\sigma(1)}}$, we find that 
\[1\leq \lambda_2\leq \lambda_3\leq \dots \leq \lambda_\ell.\]Fix real parameters
\[
1\le R_2\le \dots\le R_\ell.
\]
For a permutation $\sigma\in S_\ell$,
define $\mathcal G_\sigma(Y;R_2,\dots,R_\ell)$
to be the set of all $(g_1,\dots,g_\ell)\in\mathbb R_{\ge1}^\ell$ such that
\begin{enumerate}
\item $D_{\sigma(1)}
\leq 
D_{\sigma(2)}
\leq 
\dots
\leq 
D_{\sigma(\ell)}$,
\item $\displaystyle \prod_{j=1}^{\ell} g_j < Y$,
\item $\lambda_2\leq \lambda_3\leq \dots \leq \lambda_\ell\le R_\ell$,
\item $R_j\le \lambda_{j}$ for $2\le j\le \ell-1$.
\end{enumerate}

\subsection{Volume computations}
We introduce the following notation:
\begin{equation}\label{notation of G etc}
\begin{aligned}
\mathcal G(Y;R_2,\dots,R_\ell)
&:=
\bigcup_{\sigma\in S_{\ell}}
\mathcal G_\sigma(Y;R_2,\dots,R_\ell), \\
\mathcal G_{\mathbb Z,\sigma}(Y;R_2,\dots,R_\ell)
&:=
\mathcal G_\sigma(Y;R_2,\dots,R_\ell)\cap \mathbb Z^\ell, \\
\mathcal G_{\mathbb Z}(Y;R_2,\dots,R_\ell)
&:=
\mathcal G(Y;R_2,\dots,R_\ell)\cap \mathbb Z^\ell, \\
N_\sigma(Y;R_2,\dots,R_\ell)
&:=
\#\mathcal G_{\mathbb Z,\sigma}(Y;R_2,\dots,R_\ell), \\
N(Y;R_2,\dots,R_\ell)
&:=
\#\mathcal G_{\mathbb Z}(Y;R_2,\dots,R_\ell).
\end{aligned}
\end{equation}

\medskip

In this section we estimate the quantities
\[
N_\sigma(Y;R_2,\dots,R_\ell)
\quad\text{and}\quad
N(Y;R_2,\dots,R_\ell)
\]
as $Y\to\infty$. First, for each \[\sigma\in S_{\ell}=\op{Aut}\left(\{1, 2,\dots, \ell\}\right)\]we compute the Euclidean volume of the region 
\(
\mathcal G_\sigma(Y;R_2,\dots,R_\ell).
\)
Second, we apply Davenport's lemma to deduce an asymptotic formula for
\(
N_\sigma(Y;R_2,\dots,R_\ell)
\)
from the corresponding volume estimate. For this purpose it is necessary to obtain uniform bounds for the volumes of all coordinate projections of 
\(
\mathcal G_\sigma(Y;R_2,\dots,R_\ell)
\)
onto lower-dimensional coordinate subspaces. 
\par First we prove a number of preliminary results leading up to our volume computations. For $i=0,\dots,\ell$, define the quadratic character $\chi_i:G\rightarrow \{\pm 1\}$ by:
\[
\chi_i(\mathbf v_j)
=
(-1)^{\widetilde{\mathbf v_j\cdot \mathbf v_i}},
\]
and for $i=0,\dots,\ell$ set
\[
\mathbf w_i
:=
\bigl(
\chi_i(\mathbf v_0),
\chi_i(\mathbf v_1),
\dots,
\chi_i(\mathbf v_\ell)
\bigr)
\in \R^{\ell+1}.
\] Note that $\mathbf w_0=(1,1,\dots, 1)$ since $\mathbf v_0=0$. The Euclidean space $\mathbb{R}^{\ell+1}$ is equipped with the standard inner product $\langle\cdot, \cdot\rangle$. 

\begin{lemma}\label{lemma:character-basis}
With the notation above:

\begin{enumerate}
\item The vectors $\mathbf w_0,\mathbf w_1,\dots,\mathbf w_\ell$
form an orthogonal basis of $\R^{\ell+1}$.

\item For $i\ge 1$, the vectors $\mathbf w_i$ lie in the hyperplane
\[
H
:=
\Bigl\{
x\in\R^{\ell+1} :
\sum_{j=0}^{\ell} x_j = 0
\Bigr\},
\]

\item $\{\mathbf w_1,\dots,\mathbf w_\ell\}$
is a basis of $H$.
\end{enumerate}
\end{lemma}

\begin{proof}
The orthogonality relations for characters of the finite abelian group
$G=\F_2^n$ state that for $i,k\in\{0,\dots,\ell\}$,
\[
\sum_{j=0}^{\ell}
\chi_i(\mathbf v_j)\chi_k(\mathbf v_j)
=
\begin{cases}
2^n & \text{if } i=k, \\
0 & \text{if } i\ne k.
\end{cases}
\]
Thus the vectors $\mathbf w_0,\dots,\mathbf w_\ell$
are mutually orthogonal and nonzero, hence form a basis of $\R^{\ell+1}$.
This proves (1).

For $i\ge1$, the character $\chi_i$ is nontrivial, and hence
\[
\sum_{j=0}^{\ell} \chi_i(\mathbf v_j)=0.
\]
Thus $\mathbf w_i\in H$, proving (2).

Since $\dim H = \ell$ and there are exactly $\ell$
vectors $\mathbf w_1,\dots,\mathbf w_\ell$,
which are linearly independent by (1),
they form a basis of $H$.
This proves (3).
\end{proof}

\begin{lemma}\label{lemma:difference-basis}
Let $\sigma\in S_\ell$. For $i=2, \dots, \ell$ let $\varepsilon^i\in \R^\ell$ be the vector obtained by
deleting the $0$-th coordinate from $\mathbf w_{\sigma(i)}-\mathbf w_{\sigma(1)}$ and scaling by $1/2$, i.e.
\[
\varepsilon^i
:=
\frac{1}{2}\cdot\bigl(
\chi_{\sigma(i)}(\mathbf v_j)-\chi_{\sigma(1)}(\mathbf v_j)
\bigr)_{1\le j\le \ell}.
\]
Let
\[
H'
:=
\Bigl\{
(y_1,\dots,y_\ell)\in\R^\ell
:\ \sum_{j=1}^\ell y_j=0
\Bigr\}.
\]
Then:
\begin{enumerate}
\item Each $\varepsilon^i$ lies in $H'$.

\item The family $\{\varepsilon^i : i=2,\dots,\ell\}$
is linearly independent.

\item Consequently,
$\{\varepsilon^i : i=2,\dots,\ell\}$
forms a basis of $H'$.
\end{enumerate}
\end{lemma}

\begin{proof}
First assume that $\sigma$ is the identity permutation. Recall that $\mathbf v_0=0$, hence
\[
\chi_i(\mathbf v_0)=1
\quad\text{for all } i.
\]
Therefore the $0$-th coordinate of $\mathbf w_i-\mathbf w_1$ equals
\[
\chi_i(\mathbf v_0)-\chi_1(\mathbf v_0)=1-1=0.
\]
Thus deleting the $0$-th coordinate simply removes a common zero
coordinate from each difference vector. To prove (1), note that for $i\geq 1$,
\[
\sum_{j=0}^\ell \chi_i(\mathbf v_j)=0,
\]
and therefore, for $i\geq 2$,
\[
\sum_{j=0}^\ell
\bigl(
\chi_i(\mathbf v_j)-\chi_1(\mathbf v_j)
\bigr)=0.
\]
Because the $0$-th coordinate of the difference is zero,
this reduces to
\[
\sum_{j=1}^\ell \varepsilon^i_j=0.
\]
Hence for $i\geq 2$, we have that $\varepsilon^i\in H'$.
\par Next we prove part (2). Suppose that $\sum_{i=2}^\ell c_i \varepsilon^i=0$. Reinstating the deleted $0$-th coordinate (which is zero for every
$\varepsilon^i$), this yields
\[
\sum_{i=2}^\ell c_i(\mathbf w_i-\mathbf w_1)=0
\quad
\text{in }\R^{\ell+1}.
\]
Hence
\[
\sum_{i=2}^\ell c_i \mathbf w_i
-
\Bigl(\sum_{i=2}^\ell c_i\Bigr)\mathbf w_1
=
0.
\]
By Lemma~\ref{lemma:character-basis}, the vectors
$\mathbf w_0,\mathbf w_1,\dots,\mathbf w_\ell$
form a basis of $\R^{\ell+1}$.
In particular, $\mathbf w_1,\dots,\mathbf w_\ell$
are linearly independent.
Therefore all coefficients $c_i=0$ and the vectors $\varepsilon^i$ are linearly independent.
\par Part (3) follows from (2). In greater detail, the hyperplane $H'$ has dimension $\ell-1$ and there are exactly $\ell-1$ vectors
$\varepsilon^2,\dots,\varepsilon^\ell$,
which are linearly independent by (2).
Hence they form a basis of $H'$.
\par For a nontrivial $\sigma\in S_\ell$ an identical argument gives the result.
\end{proof}

\begin{proposition}\label{prop:volume-G-sigma}
Let $\ell \ge 2$, let $\sigma\in S_\ell$
and let $1\leq R_2 \le \cdots \le R_\ell$ be real numbers.
Then there exists an explicit constant $c_\ell>0$, depending only on $\ell$,
such that
\[
\op{Vol}\!\left(\mathcal G_\sigma(Y;R_2,\dots,R_\ell)\right)
=
c_\ell\, Y\, F(R_2,\dots,R_\ell),
\]
where
\[
F(R_2,\dots,R_\ell)
:=
\int_{R_{\ell-1}}^{R_\ell}
\int_{R_{\ell-2}}^{a_\ell}
\cdots
\int_{R_2}^{a_3}
\frac{da_2}{a_2}
\cdots
\frac{da_\ell}{a_\ell}.
\]
Moreover, the volume is independent of $\sigma$.
\end{proposition}

\begin{proof}
\par By definition, the region $\mathcal G_{\sigma}(Y;R_2,\dots,R_\ell)$ is described by positive
variables $g_1,\dots,g_\ell$ subject to:
\begin{enumerate}
\item[(i)] a discriminant constraint
\[
0<\prod_{j=1}^{\ell} g_j < Y,
\]
\item[(ii)] ordering constraints:
\[
1\leq \frac{D_{\sigma(2)}}{D_{\sigma(1)}}\le \frac{D_{\sigma(3)}}{D_{\sigma(1)}} \le \cdots \le \frac{D_{\sigma(\ell)}}{D_{\sigma(1)}} \le R_\ell\quad \text{and}\quad \frac{D_{\sigma(j)}}{D_{\sigma(1)}}\geq R_j\quad \text{for}\quad 2\leq j\leq \ell-1.
\]
\end{enumerate}
\noindent We introduce new variables
\[
a_1 := \prod_{j=1}^{\ell} g_j,
\qquad
a_j := \lambda_{j}=\frac{D_{\sigma(j)}}{D_{\sigma(1)}} \quad \text{for}\quad 2\le j\le \ell.
\]
Then condition (i) becomes
\[
0<a_1<Y,
\]
and condition (ii) becomes
\[
1\leq a_2 \le a_3 \le \cdots \le a_\ell \le R_\ell\quad \text{and}\quad a_j\geq R_j\quad \text{for}\quad 2\leq j\leq \ell-1.
\]
Introduce logarithmic variables
\[
x_i=\log g_i\quad \text{and}\quad
y_j=\log a_j \quad \text{for} \quad 1\le j\le \ell.
\]
Then
\[
y_1=x_1+\cdots+x_\ell,
\]
and for $j\ge2$,
\[
y_j=\sum_{i=1}^{\ell} c_{ji}x_i
\]
with
\[
\sum_{i=1}^{\ell} c_{ji}=0.
\]
Thus the map $(x_1, \dots, x_\ell)\mapsto (y_1, \dots, y_\ell)$ is linear, with matrix
\begin{equation}\label{defn of C}C= \begin{pmatrix} 1 & \cdots & 1 \\ c_{2,1} & \cdots & c_{2,\ell} \\ \vdots & & \vdots \\ c_{\ell,1} & \cdots & c_{\ell,\ell} \end{pmatrix}.
\end{equation}
Now recall that
\[\lambda_{j}=\frac{D_{\sigma(j)}}{D_{\sigma(1)}}
=
\frac{
\prod_{\substack{i\\
\mathbf v_i\cdot \mathbf v_{\sigma(j)}=1\\
\mathbf v_i\cdot \mathbf v_{\sigma(1)}=0}}
g_i}
{
\prod_{\substack{i\\
\mathbf v_i\cdot \mathbf v_{\sigma(j)}=0\\
\mathbf v_i\cdot \mathbf v_{\sigma(1)}=1}}
g_i}=\prod_{i=1}^\ell g_i^{\frac{\chi_{\sigma(j)}(\mathbf v_i)-\chi_{\sigma(1)}(\mathbf v_i)}{2}},\]
from which it follows that 
\[c_{j,i}=\frac{\chi_{\sigma(j)}(\mathbf v_i)-\chi_{\sigma(1)}(\mathbf v_i)}{2}.\]
\noindent It then follows from Lemma \ref{lemma:difference-basis} that the rows of the matrix $C$ above are linearly independent and that $\det C\neq0$. 
\par Since
\[
dg_i=g_i\,dx_i,
\qquad
da_j=a_j\,dy_j,
\]
we obtain
\[
dg_1\cdots dg_\ell
=
(g_1\cdots g_\ell)\,dx_1\cdots dx_\ell,
\]
and
\[
da_1\cdots da_\ell
=
(a_1\cdots a_\ell)\,dy_1\cdots dy_\ell.
\]
Using $dy=(\det C)\,dx$, we deduce
\[
dg_1\cdots dg_\ell
=
\frac{g_1\cdots g_\ell}{a_1\cdots a_\ell}
\,|\det C|^{-1}
\, da_1\cdots da_\ell.
\]
Since $a_1=g_1\cdots g_\ell$, this simplifies to
\[
dg_1\cdots dg_\ell
=
c_\ell\,
\frac{1}{a_2\cdots a_\ell}
\, da_1\cdots da_\ell,
\]
for a constant $c_\ell>0$ depending only on $\ell$.
Thus
\[
\left|
J\!\left(\frac{\partial g_i}{\partial a_j}\right)
\right|
=
c_\ell\, (a_2\cdots a_\ell)^{-1}.
\]
\noindent Hence
\[\begin{split}
\op{Vol}(\mathcal G_\sigma(Y;R_2, \dots, R_\ell))
&=
c_\ell
\int_{R_{\ell-1}}^{R_\ell}
\!\!\cdots
\int_{R_2}^{a_3}
\int_{0}^{Y}
\frac{1}{a_2\cdots a_\ell}
\, da_1\, da_2\cdots da_\ell\\
&=
c_\ell Y
\int_{R_{\ell-1}}^{R_\ell}
\!\!\cdots
\int_{R_2}^{a_3}
\frac{da_2}{a_2}\cdots\frac{da_\ell}{a_\ell}=c_\ell Y F(R_2,\dots,R_\ell).
\end{split}\]
\end{proof}

\begin{theorem}\label{thm:lattice-count}
With respect to notation from \eqref{notation of G etc}, we have that:
\[
N(Y;R_2,\dots,R_\ell)
=
\operatorname{Vol}\!\left(
\mathcal G(Y;R_2,\dots,R_\ell)
\right)
+
O\!\left(
Y^{\frac{\ell-1}{\ell}}
\right),
\]
where the implied constant depends only on $n$ and $R_\ell$.
\end{theorem}

\begin{proof}
By Davenport's lemma \cite[Theorem on p.~180]{Davenport},
\[
\bigl|
\#\mathcal G_{\mathbb Z}(Y;R_2,\dots,R_\ell)
-
\operatorname{Vol}\!\left(
\mathcal G(Y;R_2,\dots,R_\ell)
\right)
\bigr|
=
O\!\left(
\max_{0\le m\le \ell-1} V_m
\right),
\]
where $V_m$ denotes the maximal $m$–dimensional volume of a projection of
$\mathcal G(Y;R_2,\dots,R_\ell)$
onto a coordinate hyperplane obtained by setting $\ell-m$ coordinates equal to zero. Thus it suffices to prove that $V_m \ll Y^{m/\ell}$ for $0\le m\le \ell-1$.
\par Let $(g_1,\dots,g_\ell)\in
\mathcal G(Y;R_2,\dots,R_\ell)$ and write
\[
\log \lambda_k
=
\sum_{i=1}^{\ell}
\varepsilon_{ik} \log g_i,
\qquad 2\le k\le \ell,
\]
and define
\[
\varepsilon^{k}
:=
(\varepsilon_{1k},\dots,\varepsilon_{\ell k})
\in \mathbb Z^\ell.
\]
Each $\varepsilon^{k}$ lies in the hyperplane
\[
H'
=
\left\{
x\in\mathbb R^\ell:
\sum_{i=1}^{\ell} x_i=0
\right\}
\]and by Lemma~\ref{lemma:difference-basis},
the family $\{\varepsilon^{k} : 2\le k\le \ell\}$ forms a basis of $H'$.
\par
Let $x := (\log g_1,\dots,\log g_\ell)\in \mathbb R^\ell$ and
write
\[
\bar x
:=
\frac{1}{\ell}
\Bigl(
\sum_{i=1}^{\ell} \log g_i
\Bigr)
(1,\dots,1)
\]
and set $x' := x-\bar x$. Then $x'$ lies in the hyperplane
\[
H=\Bigl\{(u_1,\dots,u_\ell)\in\mathbb R^\ell :
\sum_{i=1}^{\ell} u_i=0\Bigr\}.
\]
By construction, the vectors $\varepsilon^{2},\dots,\varepsilon^{\ell}$ lie in $H$, and by Lemma~\ref{lemma:difference-basis} they form an
$\mathbb R$-basis for $H$. Hence there exist unique real numbers $a_2,\dots,a_\ell$ such that
\[
x'
=
\sum_{k=2}^{\ell} a_k\, \varepsilon^{k}.
\]

Let $\mathcal E$ denote the $(\ell-1)\times(\ell-1)$ matrix whose
$k$th column is the coordinate vector of $\varepsilon^{(k)}$
with respect to a fixed basis of $H$.
Writing $x'$ in that same basis,
the preceding identity becomes
\[
x' = \mathcal E a,
\]
where $a=(a_2,\dots,a_\ell)^{\mathsf T}$.
By Lemma~\ref{lemma:difference-basis}, $\mathcal E$ is invertible over $\mathbb R$. Setting $D:=|\det \mathcal{E}|$, the relation $a
=
\mathcal E^{-1} x'$ implies that each coefficient $a_k$
is a $\frac{1}{\ell D}\mathbb Z$-linear combination of the
$\log \lambda_j$.

Substituting into the decomposition of $x$, we obtain
\[
\log g_j
=
\frac1\ell
\sum_{i=1}^{\ell}\log g_i
+
\sum_{k=2}^{\ell}
c_{jk}\log \lambda_k,
\]
with
$c_{jk}\in \frac{1}{\ell D}\mathbb Z$.

Multiplying by $\ell$ yields
\[
\ell \log g_j
=
\sum_{i=1}^{\ell}\log g_i
+
\sum_{k=2}^{\ell}
d_{jk}\log \lambda_k,
\quad\text{where}\quad
d_{jk}\in\frac{1}{D}\Z.
\]
\noindent Exponentiating gives the identity
\begin{equation}\label{gj-identity-final}
g_j^{\ell}
=
\Bigl(
\prod_{k=2}^{\ell}
\lambda_k^{\,d_{jk}}
\Bigr)
\left(
\prod_{i=1}^{\ell} g_i
\right).
\end{equation}
Inside $\mathcal G_{\op{id}}(Y;R_2,\dots,R_\ell)$ we have
\[
1\le \lambda_k \le R_\ell,
\qquad
\prod_{i=1}^{\ell} g_i \le Y.
\]
Thus \eqref{gj-identity-final} implies
\[
g_j^{\ell}
\le
R_\ell^{C(n)} Y,
\]
where
\[
C(n)
:=
\max_{1\le j\le\ell}
\sum_{k=2}^{\ell}
|d_{jk}|.
\]
The constant $C(n)$ depends only on $n$.

Therefore
\begin{equation}\label{g_i bound}
g_j
\ll
Y^{1/\ell},
\quad \text{for}\quad
1\le j\le\ell,
\end{equation}
with implied constant depending only on $n$ and $R_\ell$.
Fix $0\le m\le \ell-1$ and consider a projection onto $m$ coordinate axes.
Each active coordinate $g_j$ lies in an interval of length
$O(Y^{1/\ell})$.
Therefore the projected region lies inside an $m$–dimensional box of side length $O(Y^{1/\ell})$ in each coordinate, and hence
\[
V_m
\ll
Y^{m/\ell}.
\]
Taking the maximum over $0\le m\le \ell-1$ gives
\[
\max_{0\le m\le \ell-1} V_m
=
O\!\left(
Y^{(\ell-1)/\ell}
\right).
\]
\end{proof}
\begin{corollary}
We have
\[
N(Y;R_2,\dots,R_\ell)
=
c_\ell \ell!
\,Y\,F(R_2,\dots,R_\ell)
+
O\!\left(
Y^{(\ell-1)/\ell}
\right).
\]
\end{corollary}

\begin{proof}
    Since there are a total of $\ell!$ permutations of $\{1, 2, 3, \dots, \ell\}$, the result follows from Proposition \ref{prop:volume-G-sigma} and Theorem \ref{thm:lattice-count}.
\end{proof}

\subsection{Sieving}

In this section we incorporate the congruence conditions necessary to ensure that
$(g_1,\dots,g_\ell)\in \mathbb Z^\ell$,
corresponds to a multiquadratic extension of the desired type.
The lattice–point estimate of Theorem~\ref{thm:lattice-count}
counts all integer points 
$\mathcal G_{\mathbb Z}(Y;R_2,\dots,R_\ell)$.
We now restrict to those $\ell$–tuples that are
\emph{strongly carefree}, and impose additional congruence
conditions modulo $4$. The condition of being strongly carefree consists of two requirements:

\begin{enumerate}
\item[(i)] Each $g_i$ is squarefree.
\item[(ii)] For every prime $p$, at most one of the $g_i$ is divisible by $p$.
\end{enumerate}
Let $z=(z_1, \dots, z_\ell)\in (\Z/4\Z)^\ell$ be such that $z_i z_j\neq 0$ for all $i, j$. 

These conditions are local at primes.
Fix a parameter $T\geq 2$, and define
\[
m=m(T):=\prod_{p<T} p^2.
\]
\begin{definition}
An $\ell$–tuple $(g_1,\dots,g_\ell)\in \mathbb Z^\ell$
is said to be \emph{$z$-strongly carefree with respect to $m$}
if for every prime $p<T$:
\begin{enumerate}
\item[(i)] $p^2\nmid g_i$ for all $i$,
\item[(ii)] at most one $g_i$ is divisible by $p$,
\item[(iii)] $(g_1, \dots, g_\ell)$ reduces to $(z_1, \dots, z_\ell)\pmod{4}$. 
\end{enumerate}
\end{definition}
\noindent Let $\mathcal L(T)\subset \mathbb Z^\ell$ denote the set of such tuples.

\medskip

For a fixed prime $p$, we describe the allowed residue classes modulo $p^2$.
When $p$ is odd we let $\mathcal C_p\subset (\mathbb Z/p^2\mathbb Z)^\ell$ be the set of tuples satisfying:
\begin{enumerate}
\item[(i)] $g_i\not\equiv 0\pmod{p^2}$ for all $i$,
\item[(ii)] at most one coordinate is $0\pmod p$.
\end{enumerate}
\noindent On the other hand when $p=2$, let $\mathcal{C}_2$ consist of $g=(g_1, \dots, g_\ell)$ which reduce to $z=(z_1, \dots, z_\ell)$ modulo $4$.

\begin{lemma}\label{lemma:Cp-general}
For every odd prime number $p$, we have
\[
\#\mathcal C_p
=
p^{\ell-1}(p-1)^\ell (p+\ell).
\]
\end{lemma}

\begin{proof}
We count directly the tuples satisfying the defining condition, namely that at most one coordinate is divisible by $p$.
\par First consider those tuples for which no coordinate is divisible by $p$. For each coordinate there are $p^2-p = p(p-1)$ such choices. Hence the number of such tuples equals
\[
(p^2-p)^\ell = p^\ell (p-1)^\ell.
\]
\par Next we count the tuples for which exactly one coordinate is divisible by $p$. Fix an index $i$. For the $i$-th coordinate there are $p-1$ elements divisible by $p$ but not by $p^2$, while each of the remaining $\ell-1$ coordinates must be prime to $p$, giving $p^2-p$ choices for each. Thus for fixed $i$ the number of such tuples is
\[
(p-1)(p^2-p)^{\ell-1} = p^{\ell-1}(p-1)^\ell.
\]
Since $i$ may be chosen in $\ell$ ways, the total number of tuples with exactly one coordinate divisible by $p$ is
\[
\ell\, p^{\ell-1}(p-1)^\ell.
\]
\noindent Adding the two contributions gives
\[
p^\ell (p-1)^\ell + \ell\, p^{\ell-1}(p-1)^\ell
= p^{\ell-1}(p-1)^\ell (p+\ell),
\]
which proves the result.
\end{proof}
\begin{definition}
The $p$–adic density of strongly carefree tuples is
\[
\mu_p:=\frac{\#\mathcal C_p}{p^{2\ell}}
=
\begin{cases}
p^{-(\ell+1)}(p-1)^\ell(p+\ell)&\text{ if }p>2;\\
\frac{1}{4^\ell}&\text{ if }p=2.
\end{cases}
\]
\end{definition}
\noindent A straightforward expansion using
\(
(p-1)^\ell
=
p^\ell\!\left(1-\frac{\ell}{p}+O(p^{-2})\right)
\)
shows that
\[
\mu_p
=
1-\frac{\ell^2}{p^2}
+
O\!\left(\frac1{p^3}\right),
\]
and in particular
\[
1-\mu_p \ll \frac1{p^2}.
\]For a parameter $T\ge 2$, define
\[
m_T := \prod_{p<T} p^2
\quad\text{and}\quad
\mu_T := \prod_{p<T}\mu_p,
\]
and
\[
\mathcal R_T(Y)
:=
\mathcal G_{\mathbb Z}(Y;R_2,\dots,R_\ell)
\cap
\mathcal L(T),
\]
where $\mathcal L(T)$ denotes the set of integer vectors satisfying the local conditions defining $\mathcal C_p$ for all primes $p<T$.

\begin{lemma}\label{lem:finite-sieve-general}
We have
\[
\#\mathcal R_T(Y)
=
\mu_T\,
\operatorname{Vol}\!\bigl(
\mathcal G(Y;R_2,\dots,R_\ell)
\bigr)
+
O\!\left(
Y^{1-\frac{1}{\ell}}
\right).
\]
\end{lemma}

\begin{proof}
By the Chinese remainder theorem, the local conditions modulo $p^2$
for distinct primes are independent. Hence, if
\[
\mathcal C_{m_T}
:=
\{x \bmod m_T : x \bmod p^2 \in \mathcal C_p
\text{ for all } p<T\},
\]
then
\[
\#\mathcal C_{m_T}
=
\prod_{p<T} \#\mathcal C_p.
\]
Moreover,
\[
\mathcal L(T)
=
\bigcup_{x\in\mathcal C_{m_T}}
(x+m_T\mathbb Z^\ell),
\]
and the union is disjoint.

Let
\[
\mathcal G(Y)
:=
\mathcal G(Y;R_2,\dots,R_\ell).
\]
For each residue class $x\in\mathcal C_{m_T}$, we apply
Theorem~\ref{thm:lattice-count} to the lattice $m_T\mathbb Z^\ell$.
Since this lattice has covolume $m_T^\ell$, we obtain
\[
\#\bigl((x+m_T\mathbb Z^\ell)\cap\mathcal G(Y)\bigr)
=
\frac{1}{m_T^\ell}
\operatorname{Vol}(\mathcal G(Y))
+
O\!\left(
Y^{1-\frac1\ell}
\right),
\]
uniformly in $x$.

Summing over all $x\in\mathcal C_{m_T}$ yields
\[
\#\mathcal R_T(Y)
=
\frac{\#\mathcal C_{m_T}}{m_T^\ell}
\operatorname{Vol}(\mathcal G(Y))
+
O\!\left(
Y^{1-\frac1\ell}
\right).
\]
Finally, by multiplicativity,
\[
\frac{\#\mathcal C_{m_T}}{m_T^\ell}
=
\prod_{p<T}
\frac{\#\mathcal C_p}{p^{2\ell}}
=
\prod_{p<T}\mu_p
=
\mu_T,
\]
which completes the proof.
\end{proof}

\begin{theorem}\label{thm:infinite-sieve-general}
Let $\mathcal L_\infty$
denote the set of nondegenerate strongly carefree tuples and set $\mathcal G(Y):=\mathcal G(Y; R_2, \dots, R_\ell)$ and $\mathcal{R}(Y):=\mathcal G(Y)\cap\mathcal L_\infty$. Then
\[
\lim_{Y\rightarrow \infty}\left(\frac{\#\mathcal{R}(Y)}{\operatorname{Vol}\!\bigl(\mathcal G(Y)\bigr)}\right)
=\prod_p \mu_p.\]
\end{theorem}

\begin{proof}
We shall denote by $\mathcal L_\infty'$ the set of strongly carefree tuples. Note that the complement $\mathcal L_\infty'\setminus \mathcal L_\infty$ consists of strongly carefree tuples for which at least one coordinate $g_i=1$. We have
\[
\mu_p
=
1-\frac{\ell^2}{p^2}
+O(p^{-3}),
\]and thus the product $\prod_p\mu_p$ converges absolutely to a positive constant.
\par For $T\ge2$, recall that $m_T=\prod_{p<T}p^2$ and 
\[\mathcal L(T)
=
\{g\in\mathbb Z^\ell :
g\bmod p^2\in\mathcal C_p
\text{ for all } p<T\}.
\]
Then by Lemma~\ref{lem:finite-sieve-general},
\[
\#\bigl(\mathcal G(Y)\cap\mathcal L(T)\bigr)
=
\prod_{p<T} \mu_p\operatorname{Vol}(\mathcal G(Y))
+
O\!\left(Y^{1-\frac1\ell}\right).
\]Let
\[
\mathcal W_{>T}
=
\bigcup_{p\ge T}
\{g\in\mathbb Z^\ell :
g\bmod p^2\notin\mathcal C_p\}.
\]
Then
\[
\mathcal G(Y)\cap\mathcal L_\infty'
=
\bigl(\mathcal G(Y)\cap\mathcal L(T)\bigr)
\setminus
\bigl(\mathcal G(Y)\cap\mathcal W_{>T}\bigr).
\]
\noindent If $g\in\mathcal W_{>T}$,
then for some prime $p\ge T$, there exist indices $i,j$ such that $p^2 |g_i g_j$. Note that if $p^2|g_i$ then we can choose any other value of $j$. From \eqref{g_i bound}, for $(g_1, \dots, g_\ell)\in \mathcal G(Y)$, each coordinate satisfies $|g_i|\ll Y^{1/\ell}$. Then $p^2\ll Y^{2/\ell}$ and thus $p<CY^{1/\ell}$ for some constant $C>0$ which depends on $\ell$. 
Fix distinct indices $i\neq j$.
For a prime $p$ in the above range,
the conditions $p\mid g_i$ and $p\mid g_j$
define a sublattice of index $p^2$.
Applying Theorem~\ref{thm:lattice-count},
uniformly for $p\le CY^{1/\ell}$,
gives
\[
\#\{g\in\mathcal G(Y): p\mid g_i,\; p\mid g_j\}
\ll
\frac{1}{p^2}
\operatorname{Vol}(\mathcal G(Y))
+
O\!\left(Y^{1-\frac1\ell}\right).
\]
\noindent By the prime number theorem, the number of primes $p<C Y^{1/\ell}$ is $\ll \frac{Y^{1/\ell}}{\log Y}$. Summing over primes $T\le p\le CY^{1/\ell}$, we find that:
\[
\sum_{T\le p\le CY^{1/\ell}}
\#\{g\in\mathcal G(Y): p\mid g_i,\; p\mid g_j\}
\ll
\operatorname{Vol}(\mathcal G(Y))
\sum_{p\ge T}\frac1{p^2}
+
O\!\left(
\frac{Y}{\log Y}
\right).
\]
Summing over the $\binom{\ell}{2}$ pairs $(i,j)$ gives
\[
\#\bigl(\mathcal G(Y)\cap\mathcal W_{>T}\bigr)
\ll
\operatorname{Vol}(\mathcal G(Y))\sum_{p\ge T}\frac1{p^2}
+
o\!\bigl(\operatorname{Vol}(\mathcal G(Y))\bigr).
\]
Thus we have shown that 
\begin{equation}\label{liminf eqn}
\#\bigl(\mathcal G(Y)\cap\mathcal L_\infty'\bigr)
=
\prod_{p<T} \mu_p\operatorname{Vol}(\mathcal G(Y))
+
O\!\left(\sum_{p\geq T} \frac{1}{p^2}\operatorname{Vol}(\mathcal G(Y)) 
\right)
+
o\!\bigl(\operatorname{Vol}(\mathcal G(Y))\bigr).
\end{equation}
We have that:
\[
\#(\mathcal G(Y)\cap\mathcal L_\infty')
\le
\#(\mathcal G(Y)\cap\mathcal L(T)),
\]
and
\[
\#(\mathcal G(Y)\cap\mathcal L_\infty')
\ge
\#(\mathcal G(Y)\cap\mathcal L(T))
-
\#(\mathcal G(Y)\cap\mathcal W_{>T}).
\]

Dividing through by $\op{Vol}(\mathcal G(Y))$ and using
Lemma~\ref{lem:finite-sieve-general}, we obtain
\[
\frac{\#(\mathcal G(Y)\cap\mathcal L(T))}
{\op{Vol}(\mathcal G(Y))}
=
\prod_{p<T}\mu_p
+
O\!\left(Y^{-1/\ell}\right).
\]
Taking the $\limsup$ as $Y\to\infty$ in the upper bound gives
\[
\limsup_{Y\to\infty}
\frac{\#(\mathcal G(Y)\cap\mathcal L_\infty')}
{\op{Vol}(\mathcal G(Y))}
\le
\prod_{p<T}\mu_p .
\]
From \eqref{liminf eqn} we find that
\[
\liminf_{Y\to\infty}
\frac{\#(\mathcal G(Y)\cap\mathcal L_\infty')}
{\op{Vol}(\mathcal G(Y))}
\ge
\prod_{p<T}\mu_p
-
O\!\left(\sum_{p\geq T} \frac{1}{p^2}\right).
\]
Letting $T\to\infty$ in the previous two inequalities therefore yields
\[
\limsup_{Y\to\infty}
\frac{\#(\mathcal G(Y)\cap\mathcal L_\infty')}
{\op{Vol}(\mathcal G(Y))}
\le
\prod_p\mu_p
\]
and
\[
\liminf_{Y\to\infty}
\frac{\#(\mathcal G(Y)\cap\mathcal L_\infty')}
{\op{Vol}(\mathcal G(Y))}
\ge
\prod_p\mu_p .
\]
\noindent Thus the $\limsup$ and $\liminf$ coincide, and we find that
\[
\lim_{Y\to\infty}
\frac{\#(\mathcal G(Y)\cap\mathcal L_\infty')}
{\op{Vol}(\mathcal G(Y))}
=
\prod_p\mu_p.
\]
The set of degenerate tuples is clearly a sparse set and thus 
\[
\lim_{Y\to\infty}
\frac{\#(\mathcal G(Y)\cap\mathcal L_\infty)}
{\op{Vol}(\mathcal G(Y))}
=
\prod_p\mu_p.
\]
\end{proof}

\begin{corollary}\label{last cor}
    We find that 
    \[\#\mathcal{R}(Y)\sim \frac{1}{4^\ell}\times \prod_{p>2} \left(1-\frac{1}{p}\right)^\ell\left(1+\frac{\ell}{p}\right) c_\ell\ell!\,F(R_2,\dots,R_\ell) Y.\]
\end{corollary}
\begin{proof}
    The result follows directly from Theorem \ref{thm:infinite-sieve-general} and Proposition \ref{prop:volume-G-sigma}.
\end{proof}
\subsection{Main theorem}

We now deduce a counting statement for totally real multiquadratic fields
with prescribed shape parameters.

Let $\mathcal{R}(Y)$ denote the set of strongly carefree tuples
\[
\mathcal{R}(Y)=\mathcal{G}(Y)\cap\mathcal{L}_\infty .
\]
By Corollary \ref{last cor} we have
\[
\#\mathcal{R}(Y)\sim
\frac{1}{4^\ell}
\Bigl(\prod_{p>2}\left(1-\frac1p\right)^\ell
\left(1+\frac{\ell}{p}\right)\Bigr)
c_\ell\ell !\,F(R_2,\dots,R_\ell)\,Y .
\]
Suppose that $2$ is unramified in $K_n$. Then recall that the discriminant of $K_n$ is given by
\[
\Delta_{K_n}=\prod_{j=1}^{\ell}D_j.
\]

\begin{definition}\label{defn of omega}
    Let $\omega_1$ denote the number of residue classes of
$(g_1,\dots,g_\ell)\pmod{4}$ for which the associated tuple
$(a_1,\dots,a_n)$ falls into case $1$.
\end{definition}

\begin{theorem}
Fix shape parameters $(R_2,\dots,R_\ell)$ with
$1\le R_2\le R_3\le \dots \le R_\ell$. Let $\mathcal{F}(X)$ denote the
set of totally real multiquadratic extensions $K_n/\mathbb{Q}$ of degree
$2^n$ such that
\begin{enumerate}
\item $2$ is unramified in $K_n$,
\item $\Delta_{K_n}\le X$,
\item there exists a permutation $\sigma$ of $\{2,\dots,\ell\}$ for which
\[
R_j\le \lambda_{\sigma(j)} \le R_{j+1} \quad \text{for} \quad 2\le j\le \ell-1 
\quad\text{and}\quad
\lambda_{\sigma(\ell)}\le R_\ell .
\]
\end{enumerate}
Then
\[
\#\mathcal{F}(X)\sim
C_\ell\,F(R_2,\dots,R_\ell)\,X^{1/2^{n-1}},
\]
where
\[
C_\ell=
\frac{\omega_1 c_\ell\ell !}{4^\ell\,\#\mathrm{GL}_n(\mathbb{F}_2)}
\prod_{p>2}
\left(1-\frac1p\right)^\ell
\left(1+\frac{\ell}{p}\right).
\]
\end{theorem}

\begin{proof}
We set $Y:=X^{1/2^{n-1}}$. Recall that a non-degenerate strongly carefree tuple $(g_1,\dots,g_\ell)$ determines squarefree integers $(D_1, \dots, D_\ell)$
from which we obtain a totally real multiquadratic field
\[
K_n=\mathbb{Q}(\sqrt{a_1},\dots,\sqrt{a_n}),
\]
where $a_i:=D_{2^{i-1}}$. This parametrization is not injective. Indeed, a different ordered generating set 
\[(a_1', \dots, a_n')=(D_{i_1}, \dots, D_{i_n})\]
for $K_n$ corresponds to choosing a different basis $(\mathbf e_1',\dots,\mathbf e_n')$
of $V$. Such a change of basis is given by multiplication by a matrix in $\mathrm{GL}_n(\mathbb{F}_2)$. 
\par Let $\mathcal{R}_1(Y)$ denote the set consisting
of tuples lying in case $(1)$. From Corollary \ref{last cor} we find that
\[
\#\mathcal{R}_1(Y)
\sim
\frac{\omega_1}{4^\ell}\,
c_\ell \ell!
\Bigl(\prod_{p>2}\left(1-\frac1p\right)^\ell
\left(1+\frac{\ell}{p}\right)\Bigr)
F(R_2,\dots,R_\ell)\,Y.
\]
\noindent Each field counted by $\mathcal{F}(X)$ arises from exactly
$\#\mathrm{GL}_n(\mathbb{F}_2)$ tuples in $\mathcal{R}_1(Y)$. Hence
\[
\#\mathcal{F}(X)
\sim
\frac{1}{\#\mathrm{GL}_n(\mathbb{F}_2)}
\#\mathcal{R}_1\!\left(X^{1/2^{n-1}}\right)
\]
and the result follows.
\end{proof}

\bibliographystyle{alpha}
\bibliography{references}

\end{document}